\documentclass[10pt,a4paper,oneside,reqno]{amsart}
\usepackage{amsfonts}
\usepackage{amsthm}
\usepackage{xcolor}
\usepackage{amsmath}
\usepackage{booktabs}
\usepackage{amscd}
\usepackage[latin2]{inputenc}
\usepackage{t1enc}
\usepackage[boldmath]{numprint}
\usepackage[mathscr]{eucal}
\usepackage{indentfirst}
\usepackage{graphicx}
\usepackage{graphics}
\usepackage{pict2e}
\usepackage{hyperref}
\usepackage{epic}
\numberwithin{equation}{section}
\usepackage[margin=2.8cm]{geometry}
\usepackage{epstopdf} 
\usepackage{enumitem}
\usepackage{amsmath,amsthm,amssymb}
\usepackage{physics}
\usepackage{mathdots}
\usepackage{mathtools}
\usepackage{enumitem}
\usepackage{array}
\usepackage{float}
\usepackage{amsfonts}
\usepackage{amsthm}
\usepackage{amsmath}
\usepackage{amscd}
\usepackage[latin2]{inputenc}
\usepackage{t1enc}
\usepackage[mathscr]{eucal}
\usepackage{indentfirst}
\usepackage{graphicx}
\usepackage{graphics}
\usepackage{pict2e}
\usepackage{hyperref}
\usepackage{epic}
\usepackage{etoolbox}

\numberwithin{equation}{section}

\theoremstyle{plain}

 \theoremstyle{definition}

\newtheorem{?}[Th]{Problem}

\DeclareMathOperator{\sign}{sgn}

\newenvironment{theorem}[2][Theorem]{\begin{trivlist}
\item[\hskip \labelsep {\bfseries #1}\hskip \labelsep {\bfseries #2.}]}{\end{trivlist}}
\newenvironment{lemma}[2][Lemma]{\begin{trivlist}
\item[\hskip \labelsep {\bfseries #1}\hskip \labelsep {\bfseries #2.}]}{\end{trivlist}}

\makeatletter
\newcounter{qrr@oldeq}
\newcounter{qrr@oldsubeq}
\newcounter{qrr@realeq}
\renewenvironment{subequations}{%
  \refstepcounter{equation}%
  \protected@edef\theparentequation{\theequation}%
  \setcounter{parentequation}{\value{equation}}%
  \setcounter{equation}{0}%
  \def\theequation{\theparentequation\alph{equation}}%
  \ignorespaces
}{%
  \setcounter{qrr@oldeq}{\value{parentequation}}%
  \setcounter{qrr@oldsubeq}{\value{equation}}%
  \setcounter{equation}{\value{parentequation}}%
  \ignorespacesafterend
}
\newenvironment{subequations*}{%
  \setcounter{qrr@realeq}{\value{equation}}%
  \let\theparentequation\theequation%
  \patchcmd{\theparentequation}{equation}{parentequation}{}{}%
  \setcounter{parentequation}{\numexpr\value{qrr@oldeq}-1}%
  \setcounter{equation}{\value{qrr@oldsubeq}}%
  \def\theequation{\theparentequation\alph{equation}}%
  \refstepcounter{parentequation}%
  \ignorespaces
}{%
  \setcounter{qrr@oldeq}{\value{parentequation}}%
  \setcounter{qrr@oldsubeq}{\value{equation}}%
  \setcounter{equation}{\value{qrr@realeq}}%
  \ignorespacesafterend
}
\makeatother

\begin{document}

\title{ High-Frequency Instabilities of a Boussinesq-Whitham System: A Perturbative Approach}

\author{Ryan Creedon$^{1}$, Bernard Deconinck$^2$, Olga Trichtchenko$^3$ \\~\\ \fontsize{0.15}{0.15}\selectfont $^1$ Dept. of Applied Mathematics, U. of Washington,  Seattle, WA, 98105, USA (\emph{creedon@uw.edu}) \\ $^2$ Dept. of Applied Mathematics, U. of Washington,  Seattle, WA, 98105, USA (\emph{deconinc@uw.edu}) \\$^3$ Dept. of Physics and Astronomy, U. of Western Ontario, London, ON, N6A 3K7, CA (\emph{otrichtc@uwo.ca}) \footnotesize \\ ~ \\ February 9, 2021} 





\begin{abstract} We analyze the spectral stability of small-amplitude, periodic, traveling-wave solutions of a Boussinesq-Whitham system.  These solutions are shown numerically to exhibit high-frequency instabilities when subject to bounded perturbations on the real line. We use a formal perturbation method to estimate the asymptotic behavior of these instabilities in the small-amplitude regime. We compare these asymptotic results with direct numerical computations.
\end{abstract}

\maketitle
\pagestyle{myheadings}
\markright{HIGH-FREQUENCY INSTABILITIES OF A BOUSSINESQ-WHITHAM SYSTEM: A PERTURBATIVE APPROACH \hfill \hfill}

\section{Introduction}

We investigate small-amplitude, $2\pi/\kappa$-periodic, traveling waves of a Boussinesq-Whitham system proposed by Hur and Pandey \cite{hurpandey19} and Hur and Tao \cite{hurtao19}:

\begin{equation} \label{1}
\begin{aligned}
\eta_t &= -h_0u_x - (\eta u)_x, \\
u_t &= -g\mathcal{K}[\eta_x] - uu_x.
\end{aligned}
\end{equation}

\noindent In this model, $\eta(x,t)$ represents the displacement of a wave profile from its equilibrium depth $h_0$, $u(x,t)$ is the horizontal velocity along $\eta$, and $\mathcal{K}$ is a Fourier multiplier operator defined so that the linearized dispersion relation of \eqref{1} matches that of the Euler water wave problem (WWP) \cite{whitham67}. For functions $f \in \textrm{L}_{\textrm{per}}^1(-\pi/\kappa,\pi/\kappa)$, $\mathcal{K}$ is defined as

\begin{equation}
\widehat{\mathcal{K}[f]}(\kappa n) = \frac{\tanh(\kappa n h_0)}{\kappa n h_0}\widehat{f}(\kappa n), \quad n \in \mathbb{Z},
\end{equation}

\noindent where $\widehat{\cdot}$ denotes the Fourier transform of $f$:

\begin{equation} \label{1b}
\widehat{f}(k) = \frac{\kappa}{2\pi} \int_{-\pi/\kappa}^{\pi/\kappa} f(x)e^{-ikx} d x.
\end{equation}

\noindent Alternatively, $\mathcal{K}$ can be defined in physical variables as the pseudo-differential operator

\begin{equation}
\mathcal{K}[f] = \left(\frac{\tanh(h_0D)}{h_0D}\right)f,
\end{equation}

\noindent where $D = -i\partial_x$. For the remainder of this manuscript, we refer to \eqref{1} as the Hur-Pandey-Tao--Boussinesq-Whitham system, or HPT--BW for short. \\
\indent The HPT--BW system is Hamiltonian \cite{hurpandey19} with

\begin{align} \label{2}
\mathcal{H} = \frac12 \int_{-\pi/\kappa}^{\pi/\kappa} \left(  h_0u^2 + g\eta \mathcal{K}[\eta] + \eta u^2 \right) d x.
\end{align}

\noindent and non-canonical Poisson structure

\begin{align}
J = -\begin{pmatrix} 0 & \partial_x \\ \partial_x & 0 \end{pmatrix}.
\end{align}

\noindent The system has a one-parameter family of small-amplitude, $2\pi/\kappa$-periodic, traveling-wave solutions for each $\kappa>0$.
We call these solutions the Stokes waves of HPT--BW by analogy with solutions of the WWP of the same name \cite{nekrasov21, stokes1847, struik26}. In Section 2, we derive a power series expansion for HPT--BW Stokes waves in a small parameter $\varepsilon$ that scales with the amplitude of the waves.

Perturbing Stokes waves yields a spectral problem after linearizing the governing equations of the perturbations. The spectral elements of this problem define the stability spectrum of Stokes waves; see Section 3. The stability spectrum is purely continuous \cite{chicone06,kapitulapromislow13,reedsimon78}, but Floquet theory decomposes the spectrum into an uncountable union of point spectra. Each of these point spectra is indexed by the Floquet exponent \cite{deconinck06,haraguskapitula08,johnson10}.

The stability spectrum inherits quadrafold symmetry from the Hamiltonian structure of \eqref{1}, i.e., the spectrum is invariant under conjugation and negation \cite{haraguskapitula08,kapitulapromislow13}. Because of quadrafold symmetry, all elements of the stability spectrum have non-positive real component only if the stability spectrum is a subset of the imaginary axis. Therefore, HPT--BW Stokes waves are  spectrally stable only if the spectrum is on the imaginary axis.  Otherwise, the Stokes waves are spectrally unstable. \\
\indent If the aspect ratio $\kappa h_0$ is sufficiently large, both HPT--BW \cite{hurpandey19} and WWP Stokes waves \cite{benjamin67,benjaminfeir67,bridgesmielke95} have stability spectra near the origin that leave the imaginary axis for $0<|\varepsilon|\ll 1$, resulting in modulational instability. Using the Floquet-Fourier-Hill (FFH) method \cite{deconinck06}, recent numerical work by \cite{claassenjohnson18} and \cite{deconinck11} shows, respectively, that HPT--BW and WWP Stokes waves also have stability spectra away from the origin that leave the imaginary axis, regardless of $\kappa h_0$. These spectra give rise to the so-called high-frequency instabilities \cite{deconinck17}, shown schematically in Figure~\ref{fig1}.

\begin{figure}[tb]
\centering \includegraphics[width=7.5cm,height=6cm]{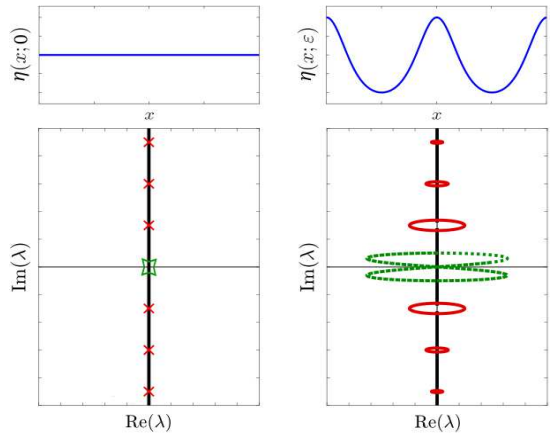}
\caption{(Top, left) A zero-amplitude ($\varepsilon = 0$) Stokes wave (solid blue). (Bottom, left) The stability spectrum of the zero-amplitude Stokes wave (solid black). Collisions of nonzero spectral elements are denoted by red crosses, while the collision of zero spectral elements is denoted by the green polygon. (Top, right) A small amplitude ($|\varepsilon| \ll 1$) Stokes wave (solid blue). (Bottom, right) The stability spectrum of the small-amplitude Stokes wave. Stable elements are in solid black. High-frequency isolas are in solid red. The modulational instability figure-eight pattern is in dashed green and is present only if $\kappa h_0$ is sufficiently large. \label{fig1} }
\end{figure}

High-frequency instabilities arise from the collision of nonzero stability eigenvalues of zero-amplitude ($\varepsilon = 0$) Stokes waves. At these collided spectral elements, a Hamiltonian-Hopf bifurcation occurs, resulting in a locus of spectral elements bounded away from the origin that leave the imaginary axis as $|\varepsilon|$ increases. We refer to this locus of spectral elements as a high-frequency isola.

High-frequency isolas are difficult to find using numerical methods like FFH. To capture the isola closest to the origin, for example, the interval of Floquet exponents that parameterizes the isola has width $\mathcal{O}\left(\varepsilon^2\right)$ (Figure~\ref{fig1b}). For isolas further from the origin, this width appears to decay geometrically in $\varepsilon$. To compound these difficulties, the isolas drift away from their initial collision sites (Figure~\ref{fig1b}), meaning that the Floquet exponent that gives rise to the collided spectral elements at $\varepsilon = 0$ is not contained in the interval parameterizing the corresponding isola for $|\varepsilon|$ sufficiently large. To circumvent these difficulties, one must supply the numerical method with asymptotic expressions for the interval of Floquet exponents corresponding to the desired isolas. We discover these expressions for high-frequency isolas of HPT--BW.

\begin{figure}[tb]
\centering \includegraphics[width=13cm,height=6cm]{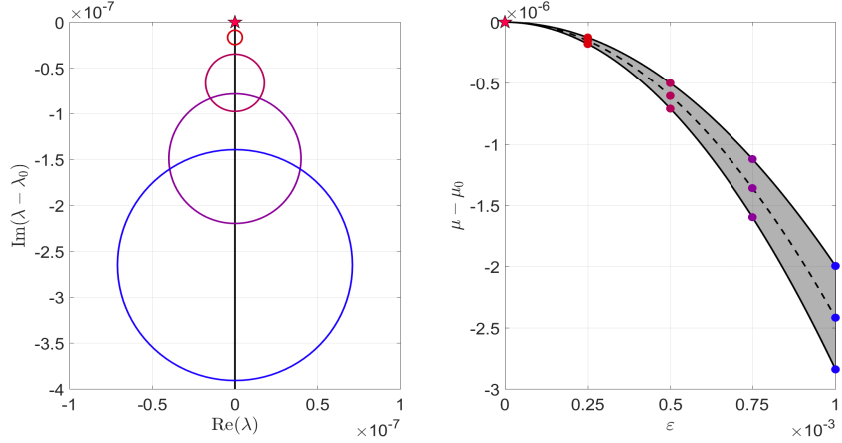}
\caption{(Left) The high-frequency isola closest to the origin of HPT--BW Stokes waves with $\kappa = g = h_0 = 1$ and $\varepsilon = 2.5 \times 10^{-4}$ (red), $5 \times 10^{-4}$ (burgundy), $7.5 \times 10^{-4}$ (purple), and $10^{-3}$  (blue). The collision that generates the isola when $\varepsilon =0$ is indicated by the red star. The imaginary axis is recentered to show the magnitude of the drift of the isola from the collision site $\lambda_0$. (Right) Interval of Floquet exponents that parameterize the isola on the left as a function of $\varepsilon$. The solid black lines indicate the boundaries of the interval, while the dashed black line gives the Floquet exponent of the most unstable spectral element of the isola. The colored dots provide the Floquet data for the correspondingly colored isola in the left plot. The red star indicates the Floquet exponent $\mu_0$ that generates the isola when $\varepsilon = 0$. The Floquet axis is recentered to show the magnitude of the drift of the Floquet exponents from $\mu_0$. \label{fig1b} }
\end{figure}

Our motivation for studying the HPT--BW system, apart from its inherent interest, is that it retains the full dispersion relation (both branches) of the more complicated WWP. Our goal is the application of the perturbation method developed herein to the WWP. The first step towards this goal was the investigation of the stability spectra of Stokes waves of the Kawahara equation \cite{creedonetal21a}. The investigations presented here constitute our second step, before proceeding to the finite-depth WWP next \cite{creedonetal21b}.

For a given high-frequency isola of an HPT--BW Stokes wave, we obtain (i) an asymptotic range of Floquet exponents that parameterize the isola, (ii) an asymptotic estimate for the most unstable spectral element of the isola, (iii) expressions of curves that are asymptotic to the isola, (iv) wavenumbers for which the given isola is not present. Our approach is inspired by a perturbation method outlined in \cite{akers15}, but modified appropriately for higher-order calculations. We compare all asymptotic results with numerical results computed by the FFH method.

\section{Small-Amplitude Stokes Waves}
In a traveling frame moving with velocity $c$, $x \rightarrow x - ct$ and \eqref{1} becomes

\begin{equation} \label{3}
\begin{aligned}
\eta_t &= c\eta_x - h_0 u_x - (\eta u)_x, \\
u_t &= cu_x - g \mathcal{K}(\eta_x) - uu_x.
\end{aligned}
\end{equation}

\noindent Non-dimensionalizing \eqref{3} according to $\eta \rightarrow h_0 \eta$, $u \rightarrow u \sqrt{gh_0}$, $x \rightarrow \alpha^{-1} h_0 x$, $t \rightarrow t \sqrt{h_0/g}$, and $c \rightarrow c\sqrt{gh_0}$ yields the following system:

\begin{equation} \label{4}
\begin{aligned}
\alpha^{-1} \eta_t &= c\eta_x - u_x - (\eta u)_x, \\
\alpha^{-1} u_t &= cu_x - \mathcal{K}_\alpha(\eta_x) - uu_x.
\end{aligned}
\end{equation}

\noindent The parameter $\alpha$ is chosen to map  $2\pi/\kappa$-periodic solutions of \eqref{3} to $2\pi$-periodic solutions of \eqref{4}. Consequently, $\alpha = \kappa h_0 > 0$, the aspect ratio of the solutions, and

\begin{align}
\widehat{\mathcal{K}_\alpha[f]}(n) = \frac{\tanh(\alpha n)}{\alpha n}\widehat{f}(n), \quad n \in \mathbb{Z},
\end{align}

\noindent or, alternatively,

\begin{align}
\mathcal{K}_\alpha[f] = \left(\frac{\tanh(\alpha D)}{\alpha D}\right)f,
\end{align}

\noindent for $f \in \textrm{L}^1_{\textrm{per}}(-\pi,\pi)$ with the Fourier transform \eqref{1b} redefined over $(-\pi,\pi)$. \\
\indent Stokes wave solutions of \eqref{4} are independent of time. Equating time derivatives in \eqref{4} to zero and integrating in $x$, we find

\begin{equation} \label{5}
\begin{aligned}
c\eta &= u + \eta u + \mathcal{I}_1 \\
cu &= \mathcal{K}_\alpha(\eta) + \frac12 u^2 + \mathcal{I}_2,
\end{aligned}
\end{equation}

\noindent where $\mathcal{I}_j$ are integration constants. For each $\alpha > 0$, there exists a three-parameter family of infinitely differentiable, even, small-amplitude,  $2\pi$-periodic solutions of \eqref{5}, provided $\mathcal{I}_j$ are sufficiently small \cite{hurpandey19}. We call these solutions the HPT--BW Stokes waves, denoted $(\eta_S(x;\varepsilon,\mathcal{I}_j),u_S(x;\varepsilon,\mathcal{I}_j))^T$, where $\varepsilon$ is a small-amplitude parameter defined implicitly in terms of the first Fourier mode of $\eta_S(x;\varepsilon,\mathcal{I}_j)$:

\begin{align} \label{5b}
\varepsilon := 2\widehat{\eta_S}(1) = \frac{1}{\pi} \int_{-\pi}^{\pi} \eta_S(x;\varepsilon,\mathcal{I}_j) e^{-ix} dx.
\end{align}

\noindent \textbf{Remark}. Redefining $c \rightarrow c - \mathcal{I}_1$ and $u \rightarrow u - \mathcal{I}_1$ in \eqref{5} implies $\mathcal{I}_1 = 0$ without loss of generality. If we also equate $\mathcal{I}_2 = 0$, our Stokes waves reduce to a one-parameter family of solutions to \eqref{5} such that \eqref{5b} ensures that $\eta_S(x;\varepsilon) \sim \varepsilon \cos(x)$ as $\varepsilon \rightarrow 0$. We restrict to this case for simplicitly, but the methodology in Sections 4 and 5 are unchanged if $\mathcal{I}_2 \neq 0$. For series representations of Stokes waves that include $\mathcal{I}_j$, see \cite{hurpandey19}. \\\\
\indent The Stokes waves and their velocity may be expanded as power series in $\varepsilon$:

\begin{subequations} \label{6}
\begin{align}
\eta_S &= \eta_S(x;\varepsilon) = \varepsilon \cos(x) + \sum_{j=2}^{\infty} \eta_j(x) \varepsilon^j,\\
u_S &= u_S(x;\varepsilon) = c_0 \varepsilon \cos(x) + \sum_{j=2}^{\infty} u_j(x) \varepsilon^j, \\
c &= c(\varepsilon) = c_0 + \sum_{j=1}^{\infty} c_{2j}\varepsilon^{2j}, \quad c_0^2 = \frac{\tanh(\alpha)}{\alpha},
\end{align}
\end{subequations}

\noindent where $\eta_j(x)$ and $u_j(x)$ are analytic, even, and $2\pi$-periodic for each $j$. Substituting these expansions into \eqref{5} (with $\mathcal{I}_j = 0$) and following a Poincar\'{e}-Lindstedt perturbation method \cite{whitham67}, one determines $\eta_j(x)$, $u_j(x)$, and $c_{2j}$ order by order. In Appendix A, we report expansions of $\eta_S$, $u_S$, and $c$ up to fourth order in $\varepsilon$; this is sufficient for our asymptotic calculations of high-frequency isolas discussed in Sections 4 and 5.

\section{The Stability Spectrum of Stokes Waves}

Consider perturbations to $(\eta_S,u_S)^T$ of the form

\begin{equation} \label{7}
\begin{pmatrix} \eta(x,t;\varepsilon) \\ u(x,t;\varepsilon) \end{pmatrix} = \begin{pmatrix} \eta_S \\ u_S\end{pmatrix} + \rho \begin{pmatrix} H(x,t;\varepsilon) \\ U(x,t;\varepsilon)\end{pmatrix} + \mathcal{O}\left(\rho^2 \right),
\end{equation}

\noindent where $|\rho| \ll 1$ is a parameter independent of $\varepsilon$ and $H$ and $U$ are sufficiently smooth, bounded functions of $x$ on $\mathbb{R}$ for all $t \geq 0$. When \eqref{7} is substituted into \eqref{4}, terms of $\mathcal{O}\left( \rho^0 \right)$ cancel by \eqref{5} (with $\mathcal{I}_j = 0$). Equating terms of $\mathcal{O}(\rho)$, the perturbation $(H,U)^T$ solves the linear system

\begin{equation}
\frac{\partial}{\partial t} \begin{pmatrix} H \\ U \end{pmatrix} = \alpha \begin{pmatrix}-u_{S}' + (c-u_S)\partial_x & - \eta_{S}' - (1+\eta_S)\partial_x \\ - \frac{i\tanh(\alpha D)}{\alpha} & - u_{S}' + (c-u_S)\partial_x \end{pmatrix} \begin{pmatrix}H \\ U \end{pmatrix},
\end{equation}

\noindent where primes denote differentiation with respect to $x$. Formally separating variables,

\begin{align} \label{7b}
\begin{pmatrix} H(x,t;\varepsilon) \\ U(x,t;\varepsilon)\end{pmatrix} = e^{\lambda t} \begin{pmatrix} \mathcal{H}(x,t;\varepsilon) \\ \mathcal{U}(x,t;\varepsilon) \end{pmatrix},
\end{align}

\noindent where $(\mathcal{H},\mathcal{U})^T$ solves the spectral problem

\begin{align} \label{8}
\lambda \begin{pmatrix} \mathcal{H} \\ \mathcal{U} \end{pmatrix} = \alpha \begin{pmatrix}-u_{S}' + (c-u_S)\partial_x & - \eta_{S}' - (1+\eta_S)\partial_x \\ - \frac{i\tanh(\alpha D)}{\alpha} & - u_{S}' + (c-u_S)\partial_x \end{pmatrix} \begin{pmatrix}\mathcal{H} \\ \mathcal{U} \end{pmatrix}.
\end{align}

\noindent Since the entries of the matrix operator above are $2\pi$-periodic, one can use Floquet theory\footnote{Strictly speaking, Floquet theory applies only to linear, local operators. Work by \cite{bronskietal16} extends this theory to nonlocal operators.} to solve \eqref{8} for $(\mathcal{H},\mathcal{U})^T$. These solutions take the form

\begin{align} \label{9}
\begin{pmatrix} \mathcal{H}(x;\varepsilon) \\ \mathcal{U}(x;\varepsilon) \end{pmatrix} = e^{i\mu x} \begin{pmatrix} \mathfrak{h}(x;\varepsilon) \\ \mathfrak{u}(x;\varepsilon)\end{pmatrix},
\end{align}

\noindent where $\mu \in [-1/2,1/2]$ is called the Floquet exponent and $\mathfrak{h}, \mathfrak{u} \in \textrm{H}^1_{\textrm{per}}(-\pi,\pi)$. Substituting \eqref{9} into \eqref{8} results in a spectral problem for $\bf{w} =  (\mathfrak{h},\mathfrak{u})^T$:

\begin{align} \label{10}
 \lambda_{\varepsilon,\mu} {\bf{w}} = {\mathcal{L}}_{\varepsilon,\mu}\bf{w},
\end{align}

\noindent with

\begin{align}
\mathcal{L}_{\varepsilon,\mu} =  \alpha \begin{pmatrix}-u_{S}' + (c-u_S)(i\mu+\partial_x) & - \eta_{S}' - (1+\eta_S)(i\mu+\partial_x) \\ - \frac{i\tanh(\alpha (\mu+D))}{\alpha} & - u_{S}' + (c-u_S)\partial_x \end{pmatrix}.
\end{align}

\noindent For sufficiently small $\varepsilon$, \eqref{10} has a countable collection of eigenvalues $\lambda_{\varepsilon,\mu}$ for each Floquet exponent $\mu$ \cite{johnson10}. The union of these eigenvalues over $\mu \in [-1/2,1/2]$ recovers the purely continuous spectrum of \eqref{8} for fixed $\varepsilon$; this is the stability spectrum of HPT--BW Stokes waves. We use the Floquet-Fourier-Hill method \cite{deconinck06} to compute the stability spectrum numerically.

If there exists a $\mu$ such that there is a $\lambda_{\varepsilon,\mu}$ with $\textrm{Re}\left(\lambda_{\varepsilon,\mu}\right) > 0$, then there exists a perturbation \eqref{7b} that grows exponentially in time, and Stokes waves of amplitude $\varepsilon$ are spectrally unstable. If no such $\mu$ is found, then the Stokes waves are spectrally stable. Because of the quadrafold symmetry mentioned in the introduction, Stokes waves are spectrally stable if and only if their stability spectrum is a subset of the imaginary axis. \\
\indent When $\varepsilon = 0$, $\mathcal{L}_{0,\mu}$ has constant coefficients, and its spectral elements are given exactly by

\begin{align}
\lambda^{(\sigma)}_{0,\mu,n} = -i\Omega_{\sigma}(n+\mu), \quad n \in \mathbb{Z}, \quad \sigma = \pm 1,
\end{align}

\noindent where $\Omega_{\sigma}$ are the two branches of the linear dispersion relation of \eqref{4} with $c \rightarrow c_0$ ($c_0$ is given in \eqref{6}). Explicitly,

\begin{align} \label{10b}
\Omega_{\sigma}(k) = -\alpha c_0 k + \sigma \omega_\alpha(k),
\end{align}

\noindent where

\begin{align}
\omega_{\alpha}(k) = \textrm{sgn}(k)\sqrt{\alpha k \tanh(\alpha k)}.
\end{align}

\noindent As expected, $\lambda_{0,\mu,n}^{(\sigma)}$ is a countable collection of eigenvalues for each $\mu$, and the resulting stability spectrum has quadrafold symmetry. In addition, the stability spectrum coincides with the imaginary axis, implying that zero-amplitude Stokes waves are spectrally stable. \\
\indent For some $\mu = \mu_0$, nonzero eigenvalues of $\mathcal{L}_{0,\mu_0}$ with double multiplicity may give rise to Hamiltonian-Hopf bifurcations and, thus, to high-frequency instabilities for $0< |\varepsilon| \ll 1$. These eigenvalues exist provided there exists $\mu_0$, $m$, and $n$ such that

\begin{align} \label{11}
\lambda^{(\sigma_1)}_{0,\mu_0,n} = \lambda^{(\sigma_2)}_{0,\mu_0,m} \neq 0.
\end{align}

\noindent We view \eqref{11} as a collision of two simple, nonzero eigenvalues. It can be shown that such a collision occurs only if $\sigma_1 \neq \sigma_2$ \cite{akersnicholls14,deconinck17,hurpandey19}. Theorem 4 in Appendix B shows that, for any $p \in \mathbb{Z} \setminus \{0,\pm 1\}$, there exist unique $\mu_0$, $m$, and $n$ that satisfy \eqref{11} with $m-n = p$. Thus, there are a countably infinite number of nonzero eigenvalue collisions in the zero-amplitude stability spectrum; each of which has potential to develop a high-frequency instability in the small-amplitude stability spectrum. \\\\
\textbf{Remark}. Using results in Appendix B, it can be shown that the Krein signatures \cite{krein51} of the colliding eigenvalues have opposite signs. This is a second necessary criterion for the occurance of high-frequency instabilities \cite{deconinck17,mackay86}. \\\\
\textbf{Remark}. The WWP shares the same collided eigenvalues with HPT--BW, since \eqref{10b} is also the dispersion relation of the WWP.

\section{High-Frequency Instabilities: $p = 2$}
We use perturbation methods to investigate the high-frequency instability that develops from the collision of $\lambda^{(1)}_{0,\mu_0,n}$ and $\lambda^{(-1)}_{0,\mu_0,m}$, where $\mu_0 \in [-1/2,1/2]$ is the unique Floquet exponent for which \eqref{11} is satisfied and\footnote{Because the spectrum \eqref{10b} has the symmetry $\overline{\lambda}^{(\sigma)}_{0,-\mu_0,-n} = \lambda^{(\sigma)}_{0,\mu_0,n}$, where the overbar denotes complex conjugation, choosing $p = -2$ gives the isola conjugate to that for $p=2$. Thus, we may choose $p = 2$ without loss of generality.} $m-n = 2$. This instability corresponds to the high-frequency isola closest to the origin; see Theorem 4 in Appendix B. For sufficiently small $\varepsilon$, this is also the isola with largest real component. \\
\subsection{The $\mathcal{O}\left(\varepsilon^0\right)$ Problem \nopunct} ~\\\\
The $p=2$ isola develops from the spectral data

\begin{subequations} \label{12}
\begin{align}
\lambda_0 &= \lambda^{(1)}_{0,\mu_0,n} = -i\Omega_{1}(\mu_0+n) = -i\Omega_{-1}(\mu_0+m) = \lambda^{(-1)}_{0,\mu_0,m} \neq 0, \\
{\bf{w_0}}(x) &= \begin{pmatrix} \mathfrak{h}_0(x)  \\ \mathfrak{u}_0(x) \end{pmatrix} = \gamma_0 \begin{pmatrix}1\\-\frac{\omega_\alpha(m+\mu_0)}{\alpha(m+\mu_0)}\end{pmatrix} e^{imx} + \gamma_1 \begin{pmatrix} 1 \\ \frac{\omega_\alpha(n+\mu_0)}{\alpha(n+\mu_0)} \end{pmatrix} e^{inx},
\end{align}
\end{subequations}

\noindent where $\gamma_j$ are arbitrary, nonzero constants. As $|\varepsilon|$ increases, we assume the spectral data vary analytically \cite{akersnicholls14} with $\varepsilon$:

\begin{subequations} \label{13}
\begin{align}
\lambda &= \lambda_0 + \varepsilon \lambda_1 + \varepsilon^2 \lambda_2 + \mathcal{O}\left( \varepsilon^3 \right), \\
{\bf{w}} &= {\bf{w_0}} + \varepsilon {\bf{w_1}} + {\varepsilon^2 \bf{w_2}} + \mathcal{O}\left(\varepsilon^3 \right) \\
&= \begin{pmatrix} \mathfrak{h}_0 \\ \mathfrak{u}_0\end{pmatrix} + \varepsilon \begin{pmatrix} \mathfrak{h}_1 \\ \mathfrak{u}_1\end{pmatrix} + \varepsilon^2 \begin{pmatrix} \mathfrak{h}_2 \\ \mathfrak{u}_2\end{pmatrix} +  \mathcal{O}\left(\varepsilon^3 \right),
\end{align}
\end{subequations}

\noindent where we suppress functional dependencies for ease of notation. We normalize $\bf{w}$ so that

\begin{align}
\widehat{\mathfrak{h}}(n)  = \frac{1}{2\pi} \int_{-\pi}^{\pi} \mathfrak{h}e^{-inx} dx = 1,
\end{align}

\noindent or, alternatively, so that

\begin{equation} \label{13b}
\begin{aligned}
\widehat{\mathfrak{h}_0}(n) &=  1, \\
\widehat{\mathfrak{h}_j}(n) &= 0, \quad \forall j \in \mathbb{N}.
\end{aligned}
\end{equation}

\noindent This normalization ensures that $\mathfrak{h}_0$ fully resolves the $n^{\textrm{th}}$ Fourier mode of $\mathfrak{h}$, a convenient choice for the perturbation calculations that follow. With this normalization,

\begin{align} \label{13c}
{\bf{w_0}}(x) = \begin{pmatrix} \mathfrak{h}_0(x)  \\ \mathfrak{u}_0(x) \end{pmatrix} = \gamma_0 \begin{pmatrix}1\\-\frac{\omega_\alpha(m+\mu_0)}{\alpha(m+\mu_0)}\end{pmatrix} e^{imx} +  \begin{pmatrix} 1 \\ \frac{\omega_\alpha(n+\mu_0)}{\alpha(n+\mu_0)} \end{pmatrix} e^{inx}.
\end{align}

\noindent The arbitrary constant $\gamma_0$ will be determined at higher order, leading to a unique expression for $\bf{w_0}$.  \\\\
\textbf{Remark}. The eigenvalue corrections $\lambda_j$ derived below are independent of the normalization chosen for $\bf{w}$.
\\\\
If $\lambda_0$ is a semi-simple, isolated eigenvalue of $\mathcal{L}_{0,\mu_0}$, we may justify \eqref{13} using analytic perturbation theory \cite{kato66}, provided the Floquet exponent is fixed. For $\varepsilon$ sufficiently small, this method of proof gives two spectral elements on the isola. Numerical and asymptotic calculations show that these spectral elements quickly leave the isola as a result of the change in its Floquet parameterization with $\varepsilon$ (Figure~\ref{fig1b}, Figure~\ref{fig4}, Figure~\ref{fig9}). To account for this variation, we allow the Floquet exponent to depend on $\varepsilon$ as well:

\begin{align} \label{14}
\mu = \mu_0 + \varepsilon \mu_1 + \varepsilon^2 \mu_2 + \mathcal{O}\left( \varepsilon^3 \right).
\end{align}


\noindent \textbf{Remark}. In the calculations that follow, explicit expressions of select quantities are suppressed for ease of readability. The interested reader may consult the supplemental Mathematica file \emph{hptbw\_isolap2.nb} for these expressions. \\

\subsection{The $\mathcal{O}\left(\varepsilon\right)$ Problem \nopunct} ~\\\\
Substituting the expansions of the Stokes wave \eqref{6}, spectral data \eqref{13}, and Floquet exponent \eqref{14} into the spectral problem \eqref{10} and collecting terms of $\mathcal{O}(\varepsilon)$, we find

\begin{align} \label{15}
\left(\mathcal{L}_{0,\mu_0} - \lambda_0\right){\bf{w_1}} = \lambda_1 {\bf{w_0}} - \textrm{$\mathcal{L}_1$} \bf{w_0},
\end{align}

\noindent with

\begin{align}
\mathcal{L}_1 = \alpha \begin{pmatrix} -u_1' + ic_0\mu_1 - u_1(i\mu_0 + \partial_x) & -\eta_1' - i\mu_1 - \eta_1(i\mu_0 + \partial_x) \\ -i\mu_1\sech^2(\alpha(\mu_0+D)) & -u_1' + ic_0\mu_1 - u_1(i\mu_0+\partial_x) \end{pmatrix}.
\end{align}

\noindent The inhomogeneous terms on the RHS of \eqref{15} can be evaluated using expressions for $\eta_1$, $u_1$, and $\bf{w_0}$. Each of these quantities are finite linear combinations of $2\pi$-periodic sinusoids. As a result, the inhomogeneous terms can be rewritten as a finite Fourier series, and \eqref{15} becomes

\begin{align} \label{16}
\left(\mathcal{L}_{0,\mu_0} - \lambda_0\right){\bf{w_1}} = \sum_{j = n-1}^{m+1} {\bf{T_{1,j}}}\textrm{$e^{ijx}$},
\end{align}

\noindent where $\bf T_{1,j}$ depend on $\mu_0$, $\alpha$, and $\gamma_0$; see the Mathematica file for details. \\\\
\textbf{Remark}. Since $m-n = 2$, the index $j \in \{n-1,n,1+n,m,m+1\}$. When evaluating the inhomogeneous terms, one finds vector multiples of exp$(i(1+n)x)$ and exp$(i(m-1)x)$. These vectors are combined to give $\bf{T_{1,1+n}}$. \\\\
\indent For \eqref{16} to have a solution $\bf{w_1}$, the inhomogeneous terms must be orthogonal (in the L$^2_{\textrm{per}}(-\pi,\pi) \times \textrm{L}^2_{\textrm{per}}(-\pi,\pi)$ sense) to the nullspace of the hermitian adjoint of $\mathcal{L}_{0,\mu_0} - \lambda_0$ by the Fredholm alternative. The hermitian adjoint of $\mathcal{L}_{0,\mu_0} - \lambda_0$ is

\begin{align} \label{17}
\left( \mathcal{L}_{0,\mu_0} - \lambda_0 \right)^\dagger =  \begin{pmatrix} - \alpha c_0(i\mu_0+\partial_x) - \overline{\lambda_0} & \tan(\alpha(i\mu_0+\partial_x))\\ \alpha (i\mu_0+\partial_x) & - \alpha c_0(i\mu_0+\partial_x) - \overline{\lambda_0} \end{pmatrix},
\end{align}

\noindent where overbars denote complex conjugation. Its nullspace is

\begin{align}
\textrm{Null}\left[ \left( \mathcal{L}_{0,\mu_0} - \lambda_0 \right)^\dagger \right] = \textrm{Span} \left[\begin{pmatrix} 1 \\ \frac{\alpha(\mu_0+n)}{\omega_\alpha(\mu_0+n)} \end{pmatrix}e^{inx},\begin{pmatrix}  1 \\ -\frac{\alpha(\mu_0+m)}{\omega_\alpha(\mu_0+m)}\end{pmatrix} e^{imx} \right].
\end{align}

\noindent Thus, according to the Fredholm alternative, there exists a solution $\bf{w_1}$ to \eqref{16} if

\begin{align}
\left< \begin{pmatrix} 1 \\ \frac{\alpha(\mu_0+n)}{\omega_\alpha(\mu_0+n)} \end{pmatrix} e^{inx}, \bf{T_{1,n}} \textrm{$e^{inx}$} \right> = 0, \quad \left< \begin{pmatrix} 1 \\ -\frac{\alpha(\mu_0+m)}{\omega_\alpha(\mu_0+m)} \end{pmatrix} e^{imx}, \bf{T_{1,m}} \textrm{$e^{imx}$} \right> = 0,
\end{align}

\noindent where $\left<\cdot,\cdot\right>$ is the standard inner product on L$^2_{\textrm{per}}(-\pi,\pi) \times \textrm{L}^2_{\textrm{per}}(-\pi,\pi)$.  Substituting expressions for $\bf{T_{1,n}}$ and $\bf{T_{1,m}}$ gives solvability conditions

\begin{subequations}
\begin{align}
\lambda_1 + i\mu_1c_{g_1}(\mu_0+n) &= 0, \\
\gamma_0\left(\lambda_1 + i \mu_1 c_{g_{-1}}(\mu_0+m)\right) &= 0,
\end{align}
\end{subequations}

\noindent where $c_{g_\sigma}(k) = \Omega'_{\sigma}(k)$ is the group velocity of $\Omega_{\sigma}$. Lemma 3 in Appendix B shows that $c_{g_{1}}(\mu_0+n) \neq c_{g_{-1}}(\mu_0+m)$. Since $\gamma_0$ is nonzero,

\begin{align} \label{18}
\lambda_1 = 0 = \mu_1.
\end{align}

\noindent Consequently, $\bf{T_{1,n}} = \bf{0} = \bf{T_{1,m}} $, simplifying the inhomogeneous terms in \eqref{16}. \\
\indent With the solvability conditions satisfied, we solve for the particular solution of $\bf{w_1}$ in \eqref{16}. Combining with the nullspace of $\mathcal{L}_{0,\mu_0} - \lambda_0$,

\begin{align}
\bf{w_1} = \textrm{$\displaystyle \sum_{\substack{j=n-1 \\ j \neq n,m}}^{m+1}$} \bf{\mathcal{W}_{1,j}} \textrm{$e^{ijx}$} + \textrm{$ \beta_{1,m} \begin{pmatrix}1\\-\frac{\omega_\alpha(m+\mu_0)}{\alpha(m+\mu_0)}\end{pmatrix} e^{imx}$} +\textrm{$\beta_{1,n}\begin{pmatrix} 1 \\ \frac{\omega_\alpha(n+\mu_0)}{\alpha(n+\mu_0)} \end{pmatrix} e^{inx},$}
\end{align}

\noindent where $\beta_{1,j}$ are arbitrary constants and ${\bf \mathcal{W}_{1,j}}$ are found in the Mathematica file. Enforcing the normalization condition \eqref{13b}, one finds $\beta_{1,n} = 0$. For ease of notation, let $\beta_{1,m} \rightarrow \gamma_1$ so that

\begin{align}
\bf{w_1} = \textrm{$\displaystyle \sum_{\substack{j=n-1 \\ j \neq n,m}}^{m+1}$} \bf{\mathcal{W}_{1,j}} \textrm{$e^{ijx}$} + \textrm{$ \gamma_1 \begin{pmatrix}1\\-\frac{\omega_\alpha(m+\mu_0)}{\alpha(m+\mu_0)}\end{pmatrix} e^{imx}$}.
\end{align}
~\\
\subsection{The $\mathcal{O}\left(\varepsilon^2\right)$ Problem \nopunct} ~\\\\
Using \eqref{18}, the spectral problem \eqref{10} at $\mathcal{O}\left(\varepsilon^2\right)$ is

\begin{align} \label{18b}
\left(\mathcal{L}_{0,\mu_0} - \lambda_0 \right){\bf{w_2}} = \lambda_2{\bf{w_0}} - \mathcal{L}_2|_{\mu_1=0} {\bf{w_0}} - \mathcal{L}_1|_{\mu_1 = 0} {\bf{w_1}},
\end{align}

\noindent where $\mathcal{L}_1|_{\mu_1 = 0}$ is the same as above, but evaluated at $\mu_1 = 0$, and

\begin{align} \label{19}
\mathcal{L}_2|_{\mu_1 = 0} = \alpha \begin{pmatrix}  -u_2' + (c_2-u_2)(i\mu_0+\partial_x) + i\mu_2 c_0 & -\eta_2' -i\mu_2 -\eta_2(i\mu_0+\partial_x) \\ -i\mu_2\sech^2(\alpha(\mu_0+D)) & -u_2' + (c_2-u_2)(i\mu_0+\partial_x) + i\mu_2 c_0                    \end{pmatrix}.
\end{align}

\noindent One can evaluate the inhomogeneous terms of \eqref{19} using $\eta_{j}$, $u_{j}$, and $\bf{w_{j-1}}$ for $j \in \{1,2\}$. These inhomogeneous terms can be expressed as a finite Fourier series, giving

\begin{align} \label{21}
\left(\mathcal{L}_{0,\mu_0} - \lambda_0 \right){\bf{w_2}} = \sum_{\substack{j=n-2 \\ j \neq n-1}}^{m+2} {\bf{T_{2,j}}}e^{ijx}.
\end{align}

\noindent It can be shown that $\bf{T_{2,n-1}} = \bf{0}$. \\
\indent Proceeding similarly to the previous order, solvability conditions for \eqref{21} are

\begin{subequations}
\begin{align}
2\left(\lambda_2 + i \mathcal{C}_{1,n}\right) + i\gamma_0\mathcal{S}_{2,n} &= 0, \label{22a} \\
2\gamma_0\left(\lambda_2 + i \mathcal{C}_{-1,m}\right)  + i \mathcal{S}_{2,m} &= 0, \label{22b}
\end{align}
\end{subequations}

\noindent where

\begin{subequations} \label{21b}
\begin{align}
\mathcal{C}_{1,n} &= \mu_2 c_{g_1}(\mu_0+n) - \mathcal{P}_{2,n}, \\
\mathcal{C}_{-1,m} &= \mu_2 c_{g_{-1}}(\mu_0+m) - \mathcal{P}_{2,m}.
\end{align}
\end{subequations}

\noindent Expressions for $\mathcal{S}_{2,j}$ and $\mathcal{P}_{2,j}$ have no dependence on $\gamma_0$, $\gamma_1$, $\mu_2$, or $\lambda_2$; see the attached Mathematica file for details. \\
\indent Conditions \eqref{22a} and \eqref{22b} form a nonlinear system for $\gamma_0$ and $\lambda_2$. Solving for $\lambda_2$ yields

\begin{align}
\lambda_2 = -i\left( \frac{\mathcal{C}_{-1,m} + \mathcal{C}_{1,n}}{2}\right) \pm \sqrt{-\left(  \frac{\mathcal{C}_{-1,m} - \mathcal{C}_{1,n}}{2} \right)^2 - \frac{\mathcal{S}_{2,n}\mathcal{S}_{2,m}}{4}}.
\end{align}

\noindent A direct calculation shows that

\begin{align} \label{22c}
\mathcal{S}_{2,n}\mathcal{S}_{2,m} = - \frac{\mathcal{S}_2^2}{\omega_{\alpha}(\mu_0+m) \omega_{\alpha}(\mu_0+n)},
\end{align}

\noindent where $\mathcal{S}_2$ is given in the attached Mathematica file. Then,

\begin{align} \label{22}
\lambda_2 = -i\left( \frac{\mathcal{C}_{-1,m} + \mathcal{C}_{1,n}}{2}\right) \pm \sqrt{-\left(  \frac{\mathcal{C}_{-1,m} - \mathcal{C}_{1,n}}{2} \right)^2 + \frac{\mathcal{S}_2^2}{4\omega_{\alpha}(\mu_0+m) \omega_{\alpha}(\mu_0+n)}}.
\end{align}

\noindent A corollary of Lemma 3 in Appendix B shows that $\omega_{\alpha}(\mu_0+m)\omega_{\alpha}(\mu_0+n)$ is positive\footnote{This corollary is equivalent to satisfying the Krein signature condition mentioned in Section 3.}. Provided $\mathcal{S}_2 \neq 0$ and $c_{g_{-1}}(\mu_0+m) \neq c_{g_1}(\mu_0+n)$, $\lambda_2$ has nonzero  real part for $\mu_2 \in (M_{2,-},M_{2,+})$, where

\begin{equation}
\begin{aligned}
M_{2,\pm} =&~ \mu_{2,*} \pm \frac{|\mathcal{S}_2|}{|c_{g_{-1}}(\mu_0+m) - c_{g_1}(\mu_0+n)|\sqrt{\omega_{\alpha}(\mu_0+m)\omega_{\alpha}(\mu_0+n)}},
\end{aligned}
\end{equation}

\noindent and

\begin{equation}
\mu_{2,*} = \frac{\mathcal{P}_{2,m} - \mathcal{P}_{2,n}}{c_{g_{-1}}(\mu_0+m) - c_{g_1}(\mu_0+n)}.
\end{equation}

\noindent That $c_{g_{-1}}(\mu_0+m) \neq c_{g_1}(\mu_0+n)$ follows from Lemma 2 in Appendix B. A plot of $\mathcal{S}_2$ as a function of $\alpha$ suggests that $\mathcal{S}_2 > 0$ for all values of $\alpha > 0$ (Figure~\ref{fig2}). {\bf We conjecture that HPT--BW Stokes waves of any wavenumer experience a $p=2$ high-frequency instability} at $\mathcal{O}\left(\varepsilon^2\right)$.

\begin{figure}[tb]
\centering \includegraphics[width=6.5cm,height=5cm]{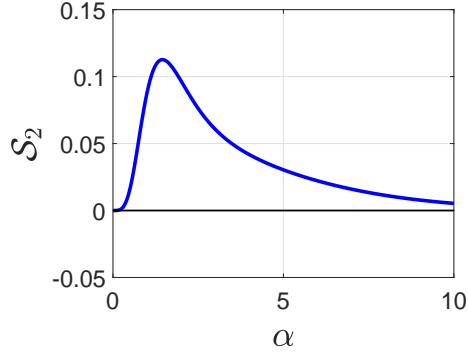}
\caption{A plot of $\mathcal{S}_2$ vs. $\alpha$. No roots of $\mathcal{S}_2$ are found for $\alpha > 0$. It is likely that HPT--BW Stokes waves of all wavenumbers experience a $p=2$ instability.}
\label{fig2}
\end{figure}

For $\mu_2 \in (M_{2,-},M_{2,+})$, a quick calculation shows that \eqref{22} parameterizes an ellipse asymptotic to the numerically observed $p=2$ high-frequency isola (Figure~\ref{fig3}). The ellipse has semi-major and -minor axes that scale with $\varepsilon^2$, and the center of the ellipse drifts along the imaginary axis like $\varepsilon^2$ from $\lambda_0$, the collision point at $\varepsilon = 0$.

\begin{figure}[tb]
\centering \includegraphics[width=12cm,height=6cm]{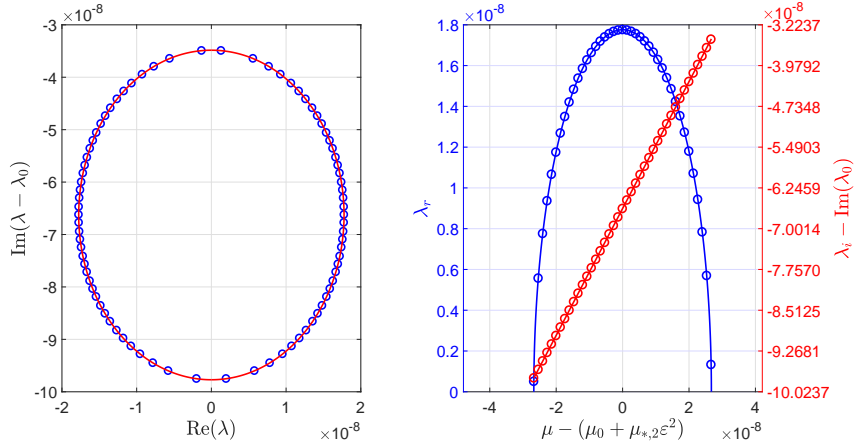}
\caption{(Left) The $p=2$ isola with $\alpha = 1$ and $\varepsilon = 5 \times 10^{-4}$ (zero-order imaginary correction removed for better visibility). The solid red curve is the ellipse obtained by our perturbation calculations. Blue circles are a subset of spectral elements from the numerically computed isola using FFH. (Right) The Floquet parameterization of the real (blue) and imaginary (red) components of the isola (zero- and second-order Floquet corrections and zero-order imaginary correction removed for better visibility). Solid curves illustrate perturbation results. Circles indicate FFH results. }
\label{fig3}
\end{figure}

The midpoint of $(M_{2,-},M_{2,+})$ maximizes the real part of $\lambda_2$. Thus, the most unstable spectral element of the isola has Floquet exponent

\begin{align}
\mu_* = \mu_0 + \mu_{2,*}\varepsilon^2 + \mathcal{O}\left(\varepsilon^3\right),
\end{align}

\noindent and its real and imaginary components are

\begin{subequations}
\begin{align}
\lambda_{r,*} &= \left( \frac{|\mathcal{S}_2|}{2\sqrt{\omega_\alpha(\mu_0+m)\omega_\alpha(\mu_0+n)}}\right) \varepsilon^2 + \mathcal{O}\left(\varepsilon^3\right), \\
\lambda_{i,*} &= -\Omega_{1}(\mu_0+n) - \mathcal{C}_{1,n}\varepsilon^2 + \mathcal{O}\left(\varepsilon^3\right),
\end{align}
\end{subequations}

\noindent respectively. These expansions agree well with the FFH results, see Figure~\ref{fig4}.

\begin{figure}[tb]
\centering \includegraphics[width=12cm,height=6cm]{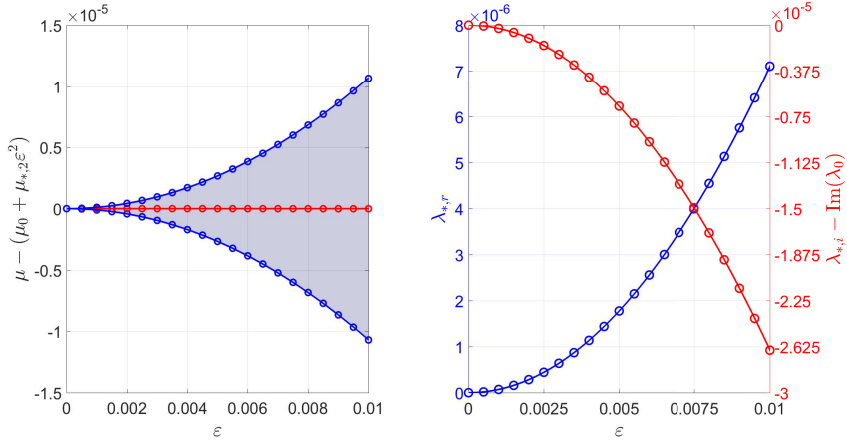}
\caption{(Left) The interval of Floquet exponents that parameterize the $p=2$ isola as a function of $\varepsilon$ with $\alpha = 1$ (zero- and second-order Floquet corrections removed for better visibility). Up to recentering the Floquet axis and accounting for a wider range of $\varepsilon$, this plot is identical to the right panel of Figure~\ref{fig1b}. Solid blue curves indicate the boundaries of this interval according to our perturbation calculations. Blue circles indicate the boundaries computed numerically by FFH. The solid red curve gives the Floquet exponent corresponding to the most unstable spectral element of the isola according to our perturbation calculations. Red circles indicate the same but computed numerically using FFH. (Right) The real (blue) and imaginary (red) components of the most unstable spectral element of the isola as a function of $\varepsilon$ (zero-order imaginary correction removed for better visibility). Solid curves illustrate perturbation calculations. Circles illustrate FFH results. }
\label{fig4}
\end{figure}

\section{High-Frequency Instabilities: $p=3$}
According to Theorem 4 in Appendix B, the $p=3$ high-frequency instability is the second-closest to the origin. As will be seen, this instability arises at $\mathcal{O}\left(\varepsilon^3\right)$. Let $\mu_0$ correspond to the unique Floquet exponent in $[-1/2,1/2]$ that satisfies the collision condition \eqref{11} with $m-n=3$. Then, the spectral data \eqref{12} give rise to the $p=3$ high-frequency instability. We assume these data and the Floquet exponent vary analytically with $\varepsilon$. For uniqueness, we normalize the eigenfunction $\bf{w}$ according to \eqref{13b} so that $\bf{w_0}$ is given by \eqref{13c}. We proceed as in the $p=2$ case. \\\\
\textbf{Remark}. In the calculations that follow, explicit expressions of select quantities are suppressed for ease of readability. The interested reader may consult the supplemental Mathematica file \emph{hptbw\_isolap3.nb} for these expressions. \\
\subsection{The $\mathcal{O}\left(\varepsilon \right)$ Problem \nopunct} ~\\\\
Substituting expansions \eqref{6}, \eqref{13}, and \eqref{14} into the spectral problem \eqref{10}, equating terms of $\mathcal{O}(\varepsilon)$, and using expression for $\eta_1$, $u_1$, and $\bf{w_0}$ to simplify, we find

\begin{align} \label{23}
\left(\mathcal{L}_{0,\mu_0} - \lambda_0\right){\bf{w_1}} = \sum_{j = n-1}^{m+1} \bf{T_{1,j}}\textrm{$e^{ijx}$}.
\end{align}

\noindent Expressions for $\bf{T_{1,j}}$ depend on $\mu_0$, $\alpha$, and $\gamma_0$; see the Mathematica file attached. Since $m-n=3$, $j \in \{n-1,n,1+n,m-1,m,m+1\}$. The functional expressions for $\bf{T_{1,n-1}}$ and $\bf{T_{1,m+1}}$ are identical to those in the $p=2$ case\footnote{They do not evaluate to the same vectors, however, as $\mu_0$ is different for $p=2$ and $p=3$ in general.}. \\
\indent Solvability conditions for \eqref{23} simplify to $\mu_1 = 0 = \lambda_1$. Together with the normalization \eqref{13b}, these conditions guarantee a solution to \eqref{23} of the form

\begin{align}
\bf{w_1} = \textrm{$\displaystyle \sum_{\substack{j=n-1 \\ j \neq n,m}}^{m+1}$} \bf{\mathcal{W}_{1,j}} \textrm{$e^{ijx}$} + \textrm{$ \gamma_{1} \begin{pmatrix}1\\-\frac{\omega_\alpha(m+\mu_0)}{\alpha(m+\mu_0)}\end{pmatrix} e^{imx}$} ,
\end{align}

\noindent where $\gamma_1$ is arbitrary and expressions for $\bf{\mathcal{W}_{1,j}}$ are found in the supplemental Mathematica file. Because $\bf{T_{1,n-1}}$ and $\bf{T_{1,m+1}}$ are identical to their $p=2$ counterparts, $\bf{\mathcal{W}_{1,n-1}}$ and $\bf{\mathcal{W}_{1,m+1}}$ are as well. \\
\subsection{The $\mathcal{O}\left(\varepsilon^2 \right)$ Problem \nopunct} ~\\\\
The $\mathcal{O}\left( \varepsilon^2 \right)$ problem takes the same form as \eqref{18b}. Evaluating at $\eta_{j}$, $u_{j}$, and $\bf{w_{j-1}}$ for $j \in \{1,2\}$, we find

\begin{align} \label{24}
\left(\mathcal{L}_{0,\mu_0} - \lambda_0 \right){\bf{w_2}} = \sum_{\substack{j=n-2 \\ j \neq n-1}}^{m+2} {\bf{T_{2,j}}}e^{ijx},
\end{align}

\noindent For the same reasons as in the $p=2$ case, $\bf{T_{2,n-1}} = \bf{0}$, and expressions for $\bf{T_{2,n-2}}$ and $\bf{T_{2,m+2}}$ are identical to their $p=2$ counterparts. \\
\indent Since $\gamma_0 \neq 0$, the solvability conditions for \eqref{24} simplify to

\begin{subequations}
\begin{align}
\lambda_2 + i\mu_2c_{g_1}(\mu_0+n) - i\mathcal{P}_{2,n} &= 0, \label{25a} \\
\lambda_2 + i\mu_2 c_{g_{-1}}(\mu_0+m) - i\mathcal{P}_{2,m} &= 0, \label{25b}
\end{align}
\end{subequations}

\noindent where $\mathcal{P}_{2,j}$ are independent of $\lambda_2$, $\mu_2$, $\gamma_0$, and $\gamma_1$; see supplemental Mathematica file. Note that these terms are distinct from those introduced in \eqref{21b}. \\
\indent Solving \eqref{25a} and \eqref{25b} for $\lambda_2$ and $\mu_2$ yields

\begin{subequations}
\begin{align}
\lambda_2 &= -i\left( \frac{\mathcal{P}_{2,m}c_{g_1}(\mu_0+n) - \mathcal{P}_{2,n}c_{g_{-1}}(\mu_0+m)               }{c_{g_{-1}}(\mu_0+m) - c_{g_1}(\mu_0+n)       } \right),  \label{25c} \\
\mu_2 &= \frac{\mathcal{P}_{2,m} - \mathcal{P}_{2,n}}{c_{g_{-1}}(\mu_0+m) - c_{g_1}(\mu_0+n)}. \label{25d}
\end{align}
\end{subequations}

\noindent Thus the spectral elements and Floquet parameterization of the $p=3$ isola have nontrivial corrections at $\mathcal{O}\left(\varepsilon^2\right)$. However, since $\textrm{Re}\left(\lambda_2\right) = 0$, we have yet to determine the leading-order behavior of the isola. We find this at the next order. \\
\indent Imposing solvability conditions \eqref{25a} and \eqref{25b} as well as the normalization condition on $\bf{w_2}$, the solution of \eqref{24} is

\begin{align} \label{26}
\bf{w_2} = \textrm{$\displaystyle \sum_{\substack{j=n-2 \\ j \neq n-1}}^{m+1}$} \bf{\mathcal{W}_{2,j}} \textrm{$e^{ijx}$} + \textrm{$ \gamma_{2} \begin{pmatrix}1\\-\frac{\omega_\alpha(m+\mu_0)}{\alpha(m+\mu_0)}\end{pmatrix} e^{imx}$} ,
\end{align}

\noindent where $\gamma_2$ is an arbitrary constant. Since $\bf{T_{2,n-1}} = \bf{0}$, $\bf{\mathcal{W}_{2,n-1}} = \bf{0}$. \\
\subsection{The $\mathcal{O}\left( \varepsilon^3 \right)$ Problem \nopunct} ~\\\\
At $\mathcal{O}\left(\varepsilon^3 \right)$, the spectral problem \eqref{10} takes the form

\begin{align} \label{27}
\left( \mathcal{L}_{0,\mu_0} - \lambda_0 \right){\bf{w_3}} = \sum_{j =2}^{3} \lambda_j {\bf{w}_{3-j}} - \sum_{j=1}^{3} \mathcal{L}_{j}|_{\mu_1 = 0} {\bf{w_{3-j}}} ,
\end{align}

\noindent where $\mathcal{L}_j|_{\mu_1 = 0}$ for $j \in \{1,2\}$ are as before and

\begin{align}
\mathcal{L}_3|_{\mu_1 = 0} = \alpha \begin{pmatrix} -u_3' - i\mu_2u_1-u_3(i\mu_0 + \partial_x) + i\mu_3 c_0 & -\eta_3'-i\mu_3 -i\eta_1\mu_2 -\eta_3(i\mu_0+\partial_x) \\ -\mu_3 \sech^2(\alpha(\mu_0+D)) & u_3' - i\mu_2u_1-u_3(i\mu_0 + \partial_x) + i\mu_3 c_0  \end{pmatrix}.
\end{align}

\noindent Evaluating \eqref{27} at $\eta_j$, $u_j$, and $\bf{w_{j-1}}$ for $j \in \{1,2,3\}$, one finds

\begin{align} \label{28}
\left(\mathcal{L}_{0,\mu_0} - \lambda_0 \right){\bf{w_3}} = \sum_{\substack{j=n-3 \\ j \neq n-2}}^{m+3} {\bf{T_{3,j}}}e^{ijx},
\end{align}

\noindent where $\bf{T_{3,n-2}} = \bf{0}$. \\
\indent The solvability conditions for \eqref{28} are

\begin{subequations}
\begin{align}
2(\lambda_3 + i\mu_3c_{g_1}(\mu_0+n)) + i\gamma_0 \mathcal{S}_{3,n} &= 0, \label{28a} \\
2\gamma_0(\lambda_3 + i\mu_3 c_{g_{-1}}(\mu_0+m)) + i\mathcal{S}_{3,m} + i \gamma_1 \mathcal{T}_{3,m} &= 0,\label{28b}
\end{align}
\end{subequations}

\noindent where $\mathcal{S}_{3,j}$ and $\mathcal{T}_{3,m}$ have no dependence on $\gamma_0$, $\gamma_1$, $\mu_3$, or $\lambda_3$; see supplemental Mathematica file. Using \eqref{25a} and \eqref{25b} from the previous order as well as \eqref{11}, one can show that $\mathcal{T}_{3,m} \equiv 0$. In addition, similar to \eqref{22c} for the $p=2$ isola, we have

\begin{align}
\mathcal{S}_{3,n}\mathcal{S}_{3,m} = - \frac{\mathcal{S}_3^2}{\omega_{\alpha}(\mu_0+m)\omega_{\alpha}(\mu_0+n)},
\end{align}

\noindent where $\mathcal{S}_3$ is given in the supplemental Mathematica file. As a result, \eqref{28a} and \eqref{28b} form a nonlinear system for $\lambda_3$ and $\gamma_0$. Solving for $\lambda_3$, one finds

\begin{align} \label{28c}
\lambda_3 =&~ -i\mu_3 \left( \frac{c_{g_{-1}}(\mu_0+m) + c_{g_1}(\mu_0+n)}{2} \right)  \\
&\pm \sqrt{-\mu_3^2 \left(\frac{c_{g_{-1}}(\mu_0+m) - c_{g_1}(\mu_0+n)}{2} \right)^2 +  \frac{\mathcal{S}_3^2}{4 \omega_{\alpha}(\mu_0+m)\omega_{\alpha}(\mu_0+n)}}. \nonumber
\end{align}

\noindent As in the $p=2$ case, $\omega_\alpha(\mu_0+m)\omega_\alpha(\mu_0+n) > 0$ and $c_{g_{-1}}(\mu_0+m) \neq c_{g_1}(\mu_0+n)$. Provided $\mathcal{S}_3 \neq 0$, $\lambda_3$ has nonzero real part if $\mu_3 \in (-M_3,M_3)$, where

\begin{align}
M_3 = \frac{|\mathcal{S}_3|}{|c_{g_{-1}}(\mu_0+m) - c_{g_1}(\mu_0+n)| \sqrt{\omega_\alpha(\mu_0+m)\omega_\alpha(\mu_0+n) }}.
\end{align}

\noindent A plot of $\mathcal{S}_3$ vs. $\alpha$ reveals that $\mathcal{S}_3=0$ only at $\alpha = 1.1862...$ (Figure~\ref{fig5}). For this wave aspect ratio, the $p=3$ instability does not occur at $\mathcal{O}\left(\varepsilon^3\right)$. In fact, Figure~\ref{fig5} shows that, if $\alpha$ approaches 1.1862.... for fixed $\varepsilon$, the numerically computed $p=3$ isola shrinks to a point on the imaginary axis.  \textbf{We conjecture that HPT--BW Stokes waves with $\alpha = 1.1862...$ are not succeptible to the $p=3$ instability}, even beyond $\mathcal{O}\left(\varepsilon^3\right)$. Indeed, in the next subsection, we find that $\lambda_4$ is purely imaginary, so Stokes waves with aspect ratio $\alpha = 1.1862...$ do not exhibit $p=3$ instabilities to $\mathcal{O}\left(\varepsilon^4\right)$.

Assuming $\alpha \neq 1.1862...$, $\mu_3 \in (-M_3,M_3)$ parameterizes an ellipse asymptotic to the $p=3$ high-frequeny isola; see Figure~\ref{fig6}. The ellipse has semi-major and -minor axes that scale with $\varepsilon^3$. The center of this ellipse drifts along the imaginary axis like $\varepsilon^2$ due to the purely imaginary correction found at $\mathcal{O}\left(\varepsilon^2\right)$.

\begin{figure}[tb]
\centering \includegraphics[width=12cm,height=6cm]{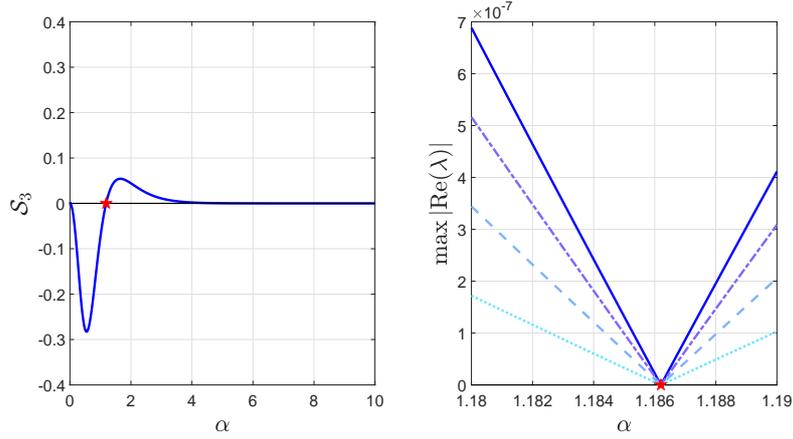}
\caption{(Left) A plot of $\mathcal{S}_3$ vs. $\alpha$. The quantity $\mathcal{S}_3$ has a root $\alpha = 1.1862...$ (red star), implying HPBT--BW Stokes waves of this aspect ratio do not have a $p=3$ instability at $\mathcal{O}\left(\varepsilon^3\right)$. (Right) A plot of the maximum real component of the numerical $p=3$ isola (computed by FFH) as a function of $\alpha$ for $\varepsilon = 10^{-3}$ (solid blue), $\varepsilon = 7.5\times 10^{-4}$ (dot-dashed purple), $\varepsilon = 5 \times 10^{-4} $ (dashed light blue), and $\varepsilon = 2.5 \times 10^{-4}$ (dotted cyan). The $p=3$ isola vanishes when $\alpha = 1.1862...$ (red star).}
\label{fig5}
\end{figure}

\begin{figure}[b!]
\centering \includegraphics[width=12cm,height=6cm]{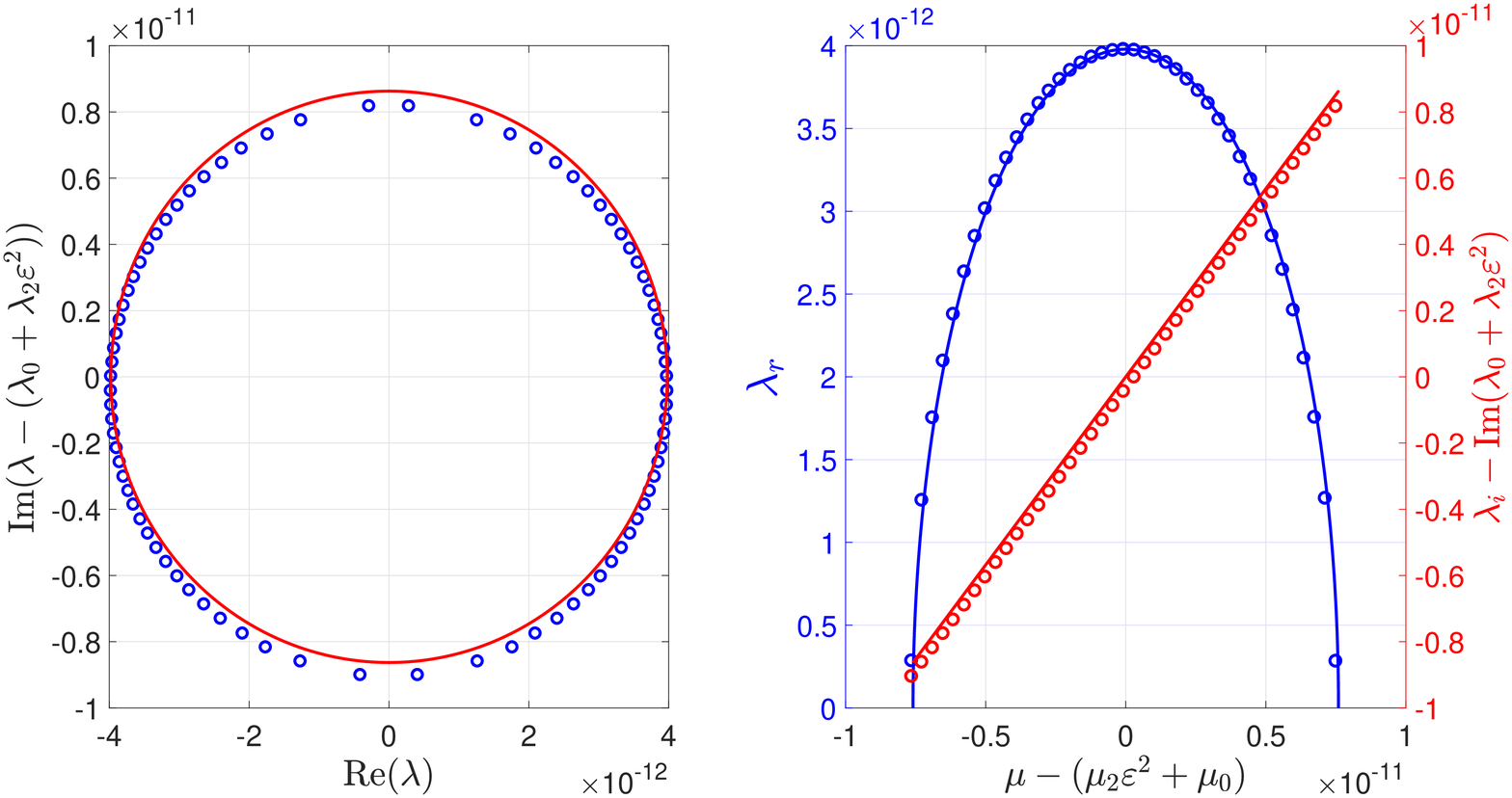}
\caption{(Left) $p=3$ isola with $\alpha = 1$ and $\varepsilon = 5 \times 10^{-4}$ (zero- and second-order imaginary corrections removed for better visibility). The solid red curve is the ellipse obtained by our perturbation calculations. The blue circles are a subset of spectral elements from the numerically computed isola using FFH. (Right) Floquet parameterization of the real (blue) and imaginary (red) components of the isola (zero- and second-order imaginary and Floquet corrections removed for better visibility). Solid curves illustrate perturbation results. Circles indicate FFH results. }
\label{fig6}
\end{figure}
~\\ \indent The interval of Floquet exponents that parameterizes the $p=3$ isola is
\begin{align} \mu \in \left(\mu_0+\mu_2\varepsilon^2 - M_3\varepsilon^3, \mu_0 + \mu_2\varepsilon^2 + M_3 \varepsilon^3\right) + \mathcal{O}\left(\varepsilon^4\right). \label{28d} \end{align}

\noindent The width of this interval is an order of magnitude smaller than that of the $p=2$ isola. Consequently, the $p=3$ isola is more challenging to find numerically than the $p=2$ isola, at least for methods similar to FFH (Table~\ref{tab1}). \\ \indent For $\alpha = 1$ and $|\varepsilon|  < 5 \times 10^{-4}$, \eqref{28d} provides an excellent approximation to the numerically computed interval of Floquet exponents (Figure~\ref{fig7}). Fourth-order corrections are necessary to improve agreement between \eqref{28d} and numerical computations for larger $\varepsilon$, see Section 5.4 below.
\begin{table}[tb]
\caption{ Intervals of Floquet exponents that parameterize the $p=2$ and $p=3$ high-frequency isolas with $\varepsilon = 10^{-3}$ and $\alpha = 1/2, 1$, and $2$. The first digit for which the boundary values disagree is underlined and colored red. If a uniform mesh of Floquet exponents in $[-1/2,1/2]$ is used for numerical methods like FFH, the spacing of the mesh must be finer than $\varepsilon^{2}$ to capture the $p=2$ instability and $\varepsilon^{3}$ to capture the $p=3$ instability. The intervals vary with $\alpha$ as well, making it difficult to adapt and refine a uniform mesh to find high-frequency isolas. }
 \centering \begin{tabular}{cc}
    \toprule
           & $p=2$  \\ \midrule
 $\alpha = \frac12$ & (-0.106478\textcolor{red}{\textbf{\underline{8}}}13547533,~-0.106478\textcolor{red}{\textbf{\underline{6}}}33575956)     \vspace*{0.2cm} \\
   $\alpha = 1$ &  (-0.26090\textcolor{red}{\textbf{\underline{9}}}131823605,~-0.26090\textcolor{red}{\textbf{\underline{8}}}917941151)       \vspace*{0.2cm} \\
   $\alpha = 2$ &  (-0.330352\textcolor{red}{\textbf{\underline{1}}}96060556,~-0.330352\textcolor{red}{\textbf{\underline{2}}}75321770)     \\ \midrule
 & $p=3$ \\ \midrule
$\alpha = \frac12$ &  (-0.37544887\textcolor{red}{\textbf{\underline{7}}}009085,~-0.37544887\textcolor{red}{\textbf{\underline{5}}}412116)   \vspace*{0.2cm}  \\
$\alpha = 1$ & (0.257196721\textcolor{red}{\textbf{\underline{1}}}00572,~0.257196721\textcolor{red}{\textbf{\underline{3}}}43587)    \vspace*{0.2cm} \\
$\alpha = 2$ & (0.0440583313\textcolor{red}{\textbf{\underline{4}}}6416,~0.0440583313\textcolor{red}{\textbf{\underline{8}}}4758)    \\ \bottomrule
  \end{tabular}

\label{tab1}
\end{table}
 ~\\ \indent Choosing $\mu_3 = 0$ maximizes the real part of $\lambda_3$. Thus, the most unstable spectral element of the $p=3$ isola has Floquet exponent

\begin{align}
\mu_* = \mu_0 + \mu_2 \varepsilon^2 + \mathcal{O}\left(\varepsilon^4\right),
\end{align}

\noindent where $\mu_2$ is as in \eqref{25d}, and its real and imaginary components are

\begin{subequations}
\begin{align}
\lambda_{r,*} &= \left( \frac{|\mathcal{S}_3|}{2\sqrt{\omega_\alpha(\mu_0+m)\omega_\alpha(\mu_0+n)}} \right)\varepsilon^3 + \mathcal{O}\left(\varepsilon^4\right),  \\
\lambda_{i,*} &= -\Omega_1(\mu_0+n) - \left(\frac{\mathcal{P}_{2,m}c_{g_1}(\mu_0+n) - \mathcal{P}_{2,n}c_{g_{-1}}(\mu_0+m)}{c_{g_{-1}}(\mu_0+m) - c_{g_1}(\mu_0+n)} \right)\varepsilon^2 + \mathcal{O}\left(\varepsilon^4\right),
\end{align}
\end{subequations}

\noindent respectively. The expansion for $\lambda_{r,*}$ is in excellent agreement with numerical results using the FFH method (Figure~\ref{fig7}). As with \eqref{28d}, corrections to $\mu_*$ and $\lambda_{i,*}$ at $\mathcal{O}\left(\varepsilon^4\right)$ improve the agreement between numerical and asymptotic results for these quantities.

\begin{figure}[tb]
\centering \includegraphics[width=12cm,height=6cm]{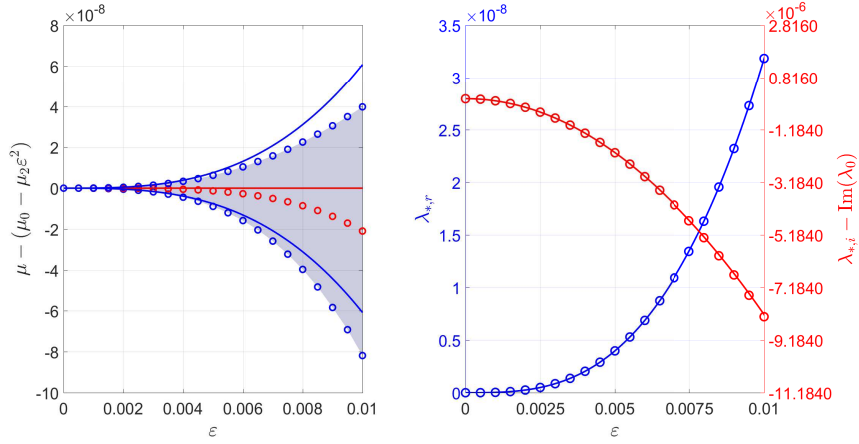}
\caption{(Left) The interval of Floquet exponents that parameterize the $p=3$ isola as a function of $\varepsilon$ with $\alpha = 1$ (zero- and second-order Floquet corrections removed for better visibility). Solid blue curves indicate the boundaries of this interval according to our perturbation calculations. Blue circles indicate the boundaries computed numerically by FFH. The solid red curve gives the Floquet exponent corresponding to the most unstable spectral element of the isola according to our perturbation calculations. Red circles indicate the same but computed numerically using FFH. (Right) The real (blue) and imaginary (red) components of the most unstable spectral element of the isola as a function of $\varepsilon$ (zero- and second-order imaginary and Floquet corrections removed for better visibility). Solid curves illustrate perturbation calculations. Circles illustrate FFH results. }
\label{fig7}
\end{figure}
\indent Before proceeding to $\mathcal{O}\left(\varepsilon^4\right)$, we solve \eqref{28} for $\bf{w_3}$, assuming solvability conditions \eqref{25a} and \eqref{25b} and normalization condition \eqref{13b} are satisfied. We find

\begin{align}
{\bf{w_3}} = \sum_{\substack{j=n-3 \\ j \neq n-2}}^{m+3} {\bf{\mathcal{W}_{3,j}}}e^{ijx} + \gamma_3 \begin{pmatrix} 1 \\ -\frac{\omega_\alpha(m+\mu_0)}{\alpha(m+\mu_0)}e^{imx} \end{pmatrix},
\end{align}

\noindent where $\gamma_3$ is arbitrary and $\bf{\mathcal{W}_{3,n-2}} = \bf{0}$ (since $\bf{T_{3,n-2}} = \bf{0}$). ~\\
\subsection{The $\mathcal{O}\left(\varepsilon^4\right)$ Problem \nopunct} ~\\\\
The spectral problem \eqref{10} is

\begin{align} \label{29}
\left( \mathcal{L}_{0,\mu_0} - \lambda_0 \right){\bf{w_4}} = \sum_{j =2}^{4} \lambda_j {\bf{w}_{3-j}} - \sum_{j=1}^{4} \mathcal{L}_{j}|_{\mu_1 = 0} {\bf{w_{3-j}}} ,
\end{align}

\noindent where $\mathcal{L}_j|_{\mu_1 = 0}$ are as before and

\begin{align}
\mathcal{L}_4|_{\mu_1 = 0} = \alpha \begin{pmatrix} \mathcal{L}_4^{(1,1)} & \mathcal{L}_4^{(1,2)} \\ \mathcal{L}_4^{(2,1)} & \mathcal{L}_4^{(1,1)} \end{pmatrix},
\end{align}

\noindent with

\begin{subequations}
\begin{align}
\mathcal{L}_4^{(1,1)} &= ic_0\mu_4 -i\mu_3u_1 + i\mu_2(c_2-u_2) + (c_4 - u_4)(i\mu_0+\partial_x) - u_4', \\
\mathcal{L}_4^{(1,2)} &= -i\mu_4 -i\mu_3\eta_1 - i\mu_2\eta_2 - \eta_4(i\mu_0+\partial_x) - \eta_4', \\
\mathcal{L}_4^{(2,1)} &= -i\mu_4 \sech^2(\alpha(\mu_0+D)) + i\alpha \mu_2^2 \sech(\alpha(\mu_0+D))\tanh(\alpha(\mu_0+D)).
\end{align}
\end{subequations}

\noindent Substituting $\eta_j$, $u_j$, and $\bf{w_{j-1}}$ for $j \in \{1,2,3\}$ into \eqref{29}, we find

\begin{align} \label{30}
\left(\mathcal{L}_{0,\mu_0} - \lambda_0 \right){\bf{w_4}} = \sum_{\substack{j=n-4\\ j \neq n-3}}^{m+4} {\bf{T_{4,j}}}e^{ijx},
\end{align}

\noindent where $\bf{T_{4,n-3}}= \bf{0}$ (since $\bf{\mathcal{W}_{3,n-2}} = \bf{0}$).  \\
\indent The solvability conditions for \eqref{30} can be expressed as

\begin{align} \label{31}
\begin{pmatrix}
2 & i\mathcal{S}_{3,n} \\ 2\gamma_0 & 2(\lambda_3 + i\mu_3c_{g_{-1}}(\mu_0+m))
\end{pmatrix}\begin{pmatrix} \lambda_4 \\ \gamma_1 \end{pmatrix} + i\gamma_2 \begin{pmatrix} 0 \\ \mathcal{T}_{4,m} \end{pmatrix} = -2i \begin{pmatrix} \mu_4c_{g_1}(\mu_0+n) - \mathcal{P}_{4,n} \\ \gamma_0\left( \mu_4 c_{g_{-1}}(\mu_0+m) - \mathcal{P}_{4,m} \right) \end{pmatrix}.
\end{align}

\noindent Expressions for $\mathcal{P}_{4,j}$ are in the supplemental Mathematica file. Using the solvability condition \eqref{25b} together with the collision condition \eqref{11} shows that $\mathcal{T}_{4,m} \equiv 0$. What remains is a linear system for $\lambda_4$ and $\gamma_1$. \\ \indent If $\alpha \neq 1.1862...$, then an application of the third-order solvability condition \eqref{28a} shows that, for $\mu_3 \in (-M_3,M_3)$,

\begin{align}
\textrm{det}\begin{pmatrix}
2 & i\mathcal{S}_{3,n} \\ 2\gamma_0 & 2(\lambda_3 + i\mu_3c_{g_{-1}}(\mu_0+m))
\end{pmatrix} = 8\lambda_{3,r},
\end{align}

\noindent where $\lambda_{3,r} = \textrm{Re}(\lambda_3)$. For $\mu_3$ in this interval, $\lambda_{3,r} \neq 0$ by construction; thus, \eqref{31} is an invertible linear system. \\
\indent We solve \eqref{31} for $\lambda_4$ by Cramer's rule, using \eqref{28a} to eliminate the dependence on $\gamma_0$. Then,

\begin{equation} \label{32}
\begin{aligned}
\lambda_4 &= i\biggr[\frac{(\lambda_3 + i\mu_3c_{g_{-1}}(\mu_0+m))(c_{g_1}(\mu_0+n) -\mathcal{P}_{4,n})}{2\lambda_{3,r}}  \phantom]
 \\
& ~~\quad \quad+ \phantom[ \frac{(\lambda_3 + i\mu_3c_{g_{1}}(\mu_0+n))(c_{g_{-1}}(\mu_0+m) -\mathcal{P}_{4,m})}{2\lambda_{3,r}}  \biggr].
\end{aligned}
\end{equation}

\noindent To simplify further, we separate the real and imaginary components of \eqref{32}. Since $\lambda_2$ \eqref{25c} is purely imaginary, $\mathcal{P}_{4,j}$ are real-valued, and $\mu_3 \in \left(-M_3,M_3\right)$, we have

\begin{align}
\lambda_{3,i} = \textrm{Im}(\lambda_3) = -i\mu_3 \left(\frac{c_{g_{-1}}(\mu_0+m) + c_{g_1}(\mu_0+n)}{2} \right),
\end{align}

\noindent according to \eqref{28c}. Equation \eqref{32} decomposes into $\lambda_4 = \lambda_{4,r} + i\lambda_{4,i}$, where

\begin{subequations}
\begin{align}
\lambda_{4,r} &= \frac{\mu_3}{4}\left[ (c_{g_{-1}}(\mu_0+m) - c_{g_1}(\mu_0+n))\left(\mu_4(c_{g_{-1}}(\mu_0+m) - c_{g_1}(\mu_0+n)) + \mathcal{P}_{4,n} - \mathcal{P}_{4,m} \right)\right], \label{32a}\\
\lambda_{4,i} &= -\frac12 \left[\mu_4(c_{g_{-1}}(\mu_0+m) + c_{g_1}(\mu_0+n)) - (\mathcal{P}_{4,m} - \mathcal{P}_{4,n}) \right].
\end{align}
\end{subequations}

\noindent As $|\mu_3| \rightarrow M_3$, $\lambda_{3,r} \rightarrow 0$. If $\lambda_{4,r}$ is to remain bounded, the numerator of \eqref{32a} must vanish in this limit. Since $c_{g_{-1}}(\mu_0+m) \neq c_{g_1}(\mu_0+n)$, we must have

\begin{align}
\mu_4 = \frac{\mathcal{P}_{4,m} - \mathcal{P}_{4,n}}{c_{g_{-1}}(\mu_0+m) - c_{g_1}(\mu_0+n)}. \label{33}
\end{align}

\noindent We refer to this equality as the \textit{regular curve condition}: it ensures that the curve asymptotic to the $p=3$ isola is continuous near its intersections with the imaginary axis.  From the \textit{regular curve condition}, we get

\begin{align} \label{34}
\lambda_4 = -i\left( \frac{\mathcal{P}_{4,m}c_{g_1}(\mu_0+n) - \mathcal{P}_{4,n}c_{g_{-1}}(\mu_0+m)               }{c_{g_{-1}}(\mu_0+m) - c_{g_1}(\mu_0+n)       } \right).
\end{align}

\noindent As expected, the Floquet parameterization and imaginary component of the $p=3$ isola have a nonzero correction at $\mathcal{O}\left(\varepsilon^4\right)$. These corrections improve the agreement between numerical and asymptotic results observed at the previous order (Figure~\ref{fig8}, Figure~\ref{fig9}). No corrections to the real component of the isola are found at fourth order. \\\\
\textbf{Remark}. If $\alpha = 1.1862...$, one can show that $\lambda_3 = 0 =\mu_3$ and $\mathcal{S}_{3,n} = 0$. Applying the Fredholm alternative to \eqref{31} gives \eqref{33}. Then, $\lambda_4$ is given by \eqref{34}, and $\gamma_0 = 1$. The constant $\gamma_1$ remains arbitrary at this order for this value of $\alpha$ only.

\begin{figure}[b!]
\centering \includegraphics[width=12cm,height=6cm]{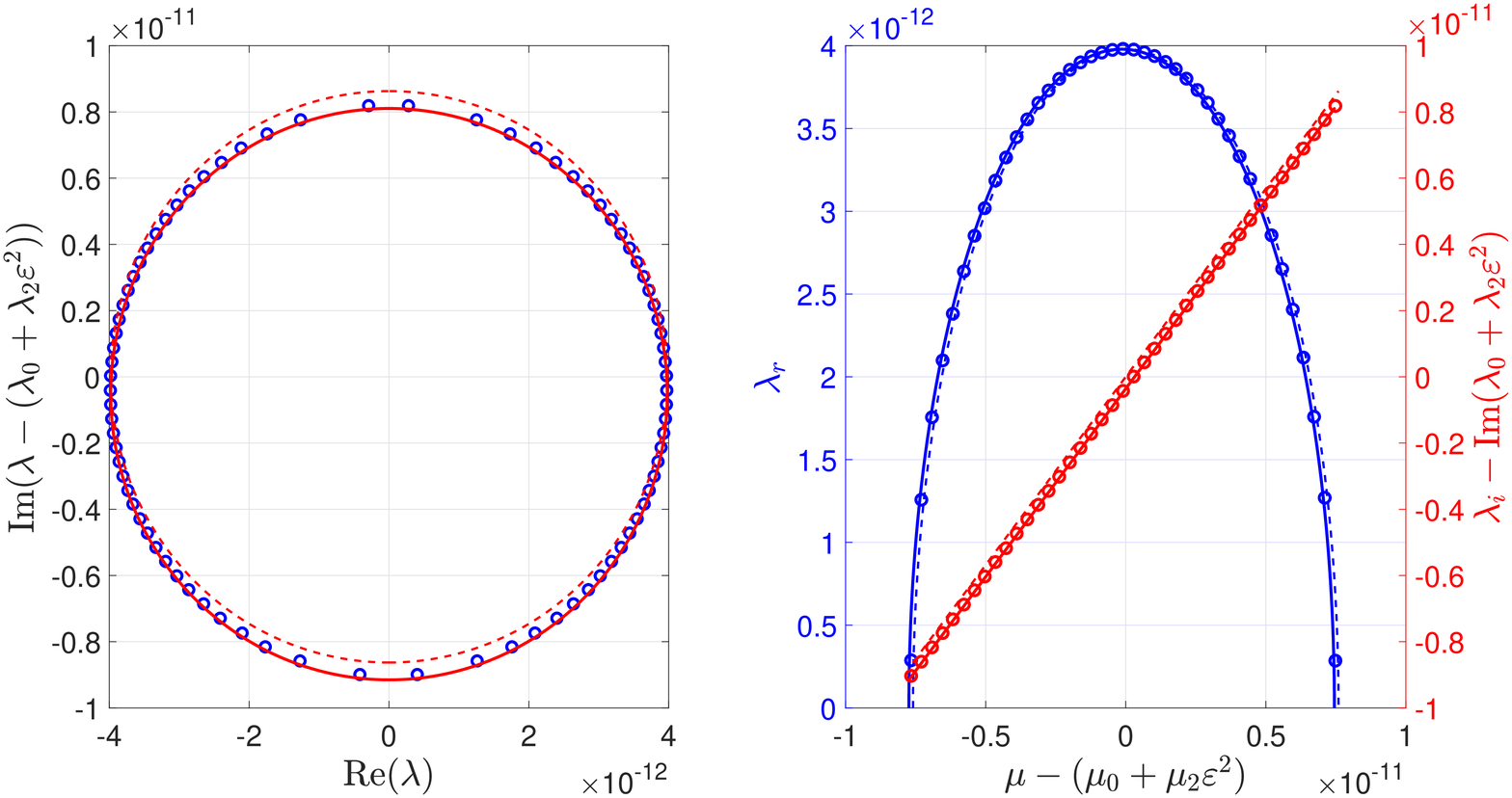}
\caption{(Left) The $p=3$ isola with $\alpha = 1$ and $\varepsilon = 5 \times 10^{-4}$ (zero- and second-order imaginary corrections removed for better visibility). Solid and dashed red curves are given by perturbation calculations to $\mathcal{O}\left(\varepsilon^4\right)$ and $\mathcal{O}\left(\varepsilon^3\right)$, respectively. Blue circles are a subset of spectral elements from the numerically computed isola using FFH. (Right) The Floquet parameterization of the real (blue) and imaginary (red) components of the isola (zero- and second-order imaginary and Floquet corrections removed for better visibility). Solid and dashed curves illustrate perturbation calculations to $\mathcal{O}\left(\varepsilon^4\right)$ and $\mathcal{O}\left(\varepsilon^3\right)$, respectively. Circles indicate FFH results. }
\label{fig8}
\end{figure}

\begin{figure}[t!]
\centering \includegraphics[width=12cm,height=6cm]{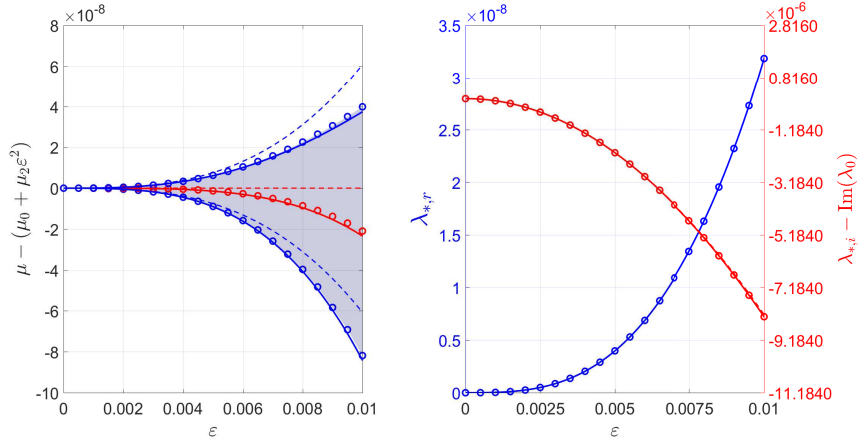}
\caption{(Left) The interval of Floquet exponents that parameterize the $p=3$ isola as a function of $\varepsilon$ with $\alpha = 1$ (zero- and second-order Floquet corrections removed for better visibility). Solid and dashed blue curves indicate the boundaries of this interval according to perturbation calculations to $\mathcal{O}\left(\varepsilon^4\right)$ and $\mathcal{O}\left(\varepsilon^3\right)$, respectively. Blue circles indicate the boundaries computed numerically by FFH. The solid red curve gives the Floquet exponent corresponding to the most unstable spectral element of the isola according to our perturbation calculations. Red circles indicate the same but computed numerically using FFH. (Right) The real (blue) and imaginary (red) components of the most unstable spectral element of the isola as a function of $\varepsilon$ (zero-order imaginary correction removed for better visibility). Solid and dashed curves illustrate perturbation calculations to $\mathcal{O}\left(\varepsilon^4\right)$ and $\mathcal{O}\left(\varepsilon^3\right)$, respectively. Circles illustrate FFH results. }
\label{fig9}
\end{figure}

\section{Conclusions}
We have extended a formal perturbation method, first introduced in \cite{akers15}, to obtain asymptotic behavior of the largest ($p=2,3$) high-frequency instabilities of small-amplitude, HPT--BW Stokes waves. In particular, we have computed explicit expressions for (i) the interval of Floquet exponents that asymptotically parameterize the $\textrm{p}^{\textrm{th}}$ isola, (ii) the leading-order behavior of its most unstable spectral elements, (iii)  the leading-order curve asymptotic to the isola, and (iv) wavenumbers that do not have a $\textrm{p}^{\textrm{th}}$ isola. Items (i)-(iii) can be extended to higher-order if necessary using the {\em regular curve condition}. In all instances, our perturbation calculations are in excellent agreement with numerical results computed by the FFH method \cite{deconinck06}. \\
\indent We restrict to $p = 2$ and $p=3$ in this work, but our method can provide asymptotic expressions for $p>3$ isolas. \textbf{We conjecture that this method yields the first real-component correction of the isola at $\mathcal{O}\left(\varepsilon^p\right)$}, similar to the cases $p = 2$ and $p = 3$. If correct, this conjecture highlights the main difficulty of computing higher-order high-frequency instabilities, both numerically and perturbatively.

The asymptotic expressions derived in this paper are intimidating and cumbersome. Although it is satisfying to have asymptotic expressions for the results previously obtained only numerically, this is not the main point of our work. Rather, (i) the perturbation method demonstrated allows one to approximate an entire isola at once, going beyond standard eigenvalue perturbation theory \cite{kato66}, (ii) the results obtained constitute a first step toward a {\em proof} of the presence of the high-frequency instabilities, and (iii) the asymptotic expressions for the range of Floquet exponents allow for a far more efficient numerical computation of the high-frequency isolas, which are difficult to track numerically as the amplitude of the solution increases.  \\\\
\textbf{Acknowledgements}: This research was funded partially by the ARCS Foundation Fellowship.
\section{Appendix A. Stokes Wave Expansions}

Below are the Stokes wave expansions of \eqref{5} (with $\mathcal{I}_j=0$) to fourth order in the small-amplitude parameter $\varepsilon$. In what follows,
\begin{align}
C_k^2 &= \frac{\tanh(\alpha k)}{\alpha k} \quad \textrm{and} \quad
D_z = \frac{1}{c_0^2 - z^2}, \quad \textrm{where} \quad c_0^2 =C_1^2.
\end{align}
\noindent For the surface displacement $\eta_S(x;\varepsilon)$,

 \begin{equation} \begin{aligned} \eta_S(x;\varepsilon) =&~ \varepsilon \eta_1(x) + \varepsilon^2\eta_2(x) + \varepsilon^3\eta_3(x) + \varepsilon^4\eta_4(x) + \mathcal{O}\left(\varepsilon^5\right) \\ =&~\varepsilon\cos(x) +\Big(N_{2,0} + 2N_{2,2}\cos(2x) \Big)\varepsilon^2 + 2N_{3,3}\varepsilon^3\cos(3x) \\ &\quad+ \Big(N_{4,0} + 2N_{4,2}\cos(2x) + 2N_{4,4}\cos(4x) \Big)\varepsilon^4 +\mathcal{O}\left(\varepsilon^5\right),  \end{aligned} \end{equation}

\noindent with
\begin{subequations}
\begin{align}
N_{2,0} =&~ \frac{3c_0^2D_1}{4}, \\
N_{2,2} =&~ \frac{3c_0^2D_{C_2}}{8}, \\
N_{3,3} =&~ \frac{c_0^2D_{C_2}D_{C_3}}{16}\Big(5c_0^2 + 4C_2^2 \Big), \\
N_{4,0} =&~ -\frac{3c_0^2D_1^3D_{C_2}^2}{64} \Big( 25c_0^8+4C_2^4+2c_0^2C_2^2(-7+2C_2^2) - 4c_0^6(2+11C_2^2) \phantom) \\
&\quad \phantom(+ c_0^4(1+22C_2^2 + 10C_2^4) \Big), \nonumber \\
N_{4,2} =&~ \frac{c_0^2D_1D_{C_2}^3D_{C_3}}{64}\Big( -20c_0^8+4C_2^4C_3^2+c_0^6(2+31C_2^2+50C_3^2)\phantom) \\
& \quad\phantom( +c_0^2C_2^2(C_3^2+2C_2^2(10+7C_3^2))  - c_0^4(38C_2^4+32C_3^2+C_2^2(-5+37C_3^2)) \Big),\nonumber \\
N_{4,4} =&~ \frac{c_0^2 D_{C_2}^2D_{C_3} D_{C_4}}{128}\Big(35c_0^6 - 20C_2^4C_3^2+5c_0^4(4C_2^2+5C_3^2) - 4c_0^2(7C_2^4+8C_2^2C_3^2)\Big).
\end{align}
\end{subequations}
\noindent For the horizontal velocity $u_S(x;\varepsilon)$ along $\eta_S(x;\varepsilon)$,

\begin{equation}\begin{aligned} u_S(x;\varepsilon) =&~ \varepsilon u_1(x) + \varepsilon^2u_2(x) + \varepsilon^3u_3(x) + \varepsilon^4u_4(x) + \mathcal{O}\left(\varepsilon^5\right) \\ =&~c_0\varepsilon\cos(x) +\Big(U_{2,0} + 2U_{2,2}\cos(2x) \Big)\varepsilon^2 + \Big(2U_{3,1}\cos(x)+2U_{3,3}\cos(3x)\Big)\varepsilon^3  \\ &\quad+ \Big(U_{4,0} + 2U_{4,2}\cos(2x) + 2U_{4,4}\cos(4x) \Big)\varepsilon^4 +\mathcal{O}\left(\varepsilon^5\right),   \end{aligned} \end{equation}

with

\begin{subequations}
\begin{align}
U_{2,0} =&~ \frac{c_0 D_1}{4}\Big(2+c_0^2\Big), \\
U_{2,2} =&~ \frac{c_0 D_{C_2}}{8}\Big(2C_2^2+c_0^2\Big), \\
U_{3,1} =&~ \frac{3c_0^3D_1D_{C_2}}{32} \Big(1-3c_0^2+2C_2^2\Big),\\
U_{3,3} =&~\frac{c_0D_{C_2}D_{C_3}}{16}\Big(c_0^4+2C_2^2C_3^2+2c_0^2(C_2^2+2C_3^2)\Big), \\
U_{4,0} =&~ \frac{-3c_0^3D_1^3D_{C_2}^2(2+c_0^2)}{64}\Big(2c_0^4(-2+5c_0^2)+C_2^2(-3+8c_0^2-17c_0^4)+2C_2^4(1+2c_0^2)\Big), \\
U_{4,2} =&~\frac{c_0D_1D_{C_2}^3D_{C_3}}{128}\Big(-25c_0^{10}+8C_2^6C_3^2 + c_0^8(7-C_2^2+45C_3^2)+c_0^6\big(C_2^2+32C_2^4 \phantom) \phantom) \\
&\quad \phantom( \phantom(+C_3^2(-27+37C_2^2)\big) -c_0^4(60C_2^6+37C_2^2C_3^2+C_2^4(-22+56C_3^2))\Big), \nonumber \\
U_{4,4} =&~\frac{c_0D_{C_2}^2D_{C_3}D_{C_4}}{128}\Big(5c_0^8-8C_2^4C_3^2C_4^2 + 2c_0^4(-2C_2^4+5C_3^2C_4^2+6C_2^2(-C_3^2+C_4^2)) \phantom) \\
&\quad \phantom( + c_0^6(8C_2^2+15(C_3^2+2C_4^2))-4c_0^2(5C_2^2C_3^2C_4^2 + 3C_2^4(C_3^2+2C_4^2)) \Big). \nonumber
\end{align}
\end{subequations}

\noindent For the velocity of the Stokes waves $c(\varepsilon)$,

\begin{align}
c(\varepsilon) = c_0 + c_2\varepsilon^2 + c_4\varepsilon^4 + \mathcal{O}\left(\varepsilon^6\right),
\end{align}

with

\begin{subequations}
\begin{align}
c_2 =&~ \frac{3c_0D_1D_{C_2}}{16}\Big(c_0^2+5c_0^4-2C_2^2(2+c_0^2) \Big), \\
c_4 =&~\frac{3c_0D_1^3D_{C_2}^3D_{C_3}}{512}\Big(3c_0^2\big(c_0^6-3c_0^8+15c_0^{10} - 85c_0^{12} + c_0^4C_2^2(11+3c_0^2-3c_0^4+205c_0^6) \phantom) \phantom) \\
&\quad \phantom(-4c_0^2C_2^4(-2+15c_0^2+3c_0^4+38c_0^6) - 4C_2^6(-4+12c_0^2-27c_0^4+c_0^6)\big) \nonumber \\
&\quad+ C_3^2\big(c_0^6(-103+309c_0^2-345c_0^4+355c_0^6) -3c_0^4C_2^2(31-57c_0^2+57c_0^4+185c_0^6)\phantom) \nonumber \\
&\quad \phantom( \phantom( + 36c_0^2C_2^4(2-3c_0^2+9c_0^4+10c_0^6) - 4C_2^6(2+c_0^2)(-2+7c_0^2+13c_0^4)\big)\Big). \nonumber
\end{align}
\end{subequations}

\section{Appendix B. Collision Condition}
Up to redefining $m$ and $n$, \eqref{11} simplifies to

\begin{align} \label{B0}
\Omega_1(\mu_0+n) = \Omega_{-1}(\mu_0+m) \neq 0.
\end{align}

\noindent With $k = \mu_0+n$ and $p=m-n$, \eqref{B0} becomes

\begin{align} \label{B8}
\Omega_1(k) = \Omega_{-1}(k+p) \neq 0.
\end{align}

\noindent We refer to \eqref{B8} as the collision condition. We prove that, for each $p \in \mathbb{Z} \setminus \{0,\pm1\}$, there exists a unique $k(p;\alpha)$ that satisfies the collision condition. These solutions $k(p;\alpha)$ are distinct from each other (for each $\alpha>0$) and result in an infinite number of distinct collision points on the imaginary axis, according to \eqref{11}. First, we establish important monotonicity properties of $\Omega_{\sigma}(k)$, defined in \eqref{10b}.
\begin{lemma}{1}  The function $\omega_\alpha(k) = \sign(k)  \sqrt{\alpha k\tanh(\alpha k)} $ is strictly increasing for $k \in \mathbb{R}$. If $|k| > 1$, then  $\omega_\alpha'(k) < \alpha |c_0|$, where $c_0^2 = \tanh(\alpha)/\alpha$.
\end{lemma}
\begin{proof} A direct calculation shows

\begin{align} \label{B9} \omega'_\alpha(k) &=  \frac{1}{2}\left(\sqrt{\frac{\alpha \tanh(\alpha k)}{k}} + \alpha  \sqrt{\frac{\alpha k}{\sinh(\alpha k)}}  \sech^{3/2}(\alpha k) \right),
\end{align}

\noindent from which $\omega'_\alpha(k) >0$. This proves the first claim. Since $\tanh(\alpha k)/(\alpha k) \leq 1$, $\alpha k/\sinh(\alpha k) \leq 1$, and $\sech(\alpha k) \leq 1$, \eqref{B9} gives

\begin{align}
\omega_{\alpha}'(k) \leq \frac{1}{2}\left(\sqrt{\frac{\alpha \tanh(\alpha k)}{k}} + \alpha \sech(\alpha k)  \right). \label{B10b}
\end{align}

\noindent Since $\alpha > 0$, $\sinh(\alpha)/\alpha > 1 > \sech(\alpha)$, so that $\sech(\alpha) < |c_0|$. Because $\sech(z)$ is even and strictly decreasing for $z>0$, we have

\begin{align} \sech(\alpha k) < |c_0|, \quad \textrm{for} \quad |k| > 1. \label{B10c}\end{align}

\noindent Similarly, since $\tanh(z)/z$ is even and strictly decreasing for $z>0$,

\begin{align} \sqrt{\frac{\alpha \tanh(\alpha k)}{k}} < \alpha \sqrt{\frac{\tanh(\alpha)}{\alpha}} = \alpha |c_0|, \quad \textrm{for} \quad |k|>1. \label{B10d} \end{align}

\noindent Together with \eqref{B10b}, inequalities \eqref{B10c} and \eqref{B10d} imply $\omega'_\alpha(k) < \alpha |c_0|$ for $|k|>1$. \end{proof}
\begin{lemma}{2}
If $c_0>0$, $\Omega_{-1}(k)$ is strictly decreasing for $k \in \mathbb{R}$, and $\Omega_{1}(k)$ is strictly decreasing for $|k|>1$. If $c_0<0$, $\Omega_1(k)$ is strictly increasing for $k \in \mathbb{R}$, and $\Omega_{-1}(k)$ is strictly increasing for $|k|>1$.
\end{lemma}
\begin{proof} Suppose $c_0>0$. By definition, $\Omega'_\sigma(k) = -\alpha c_0 + \sigma \omega'_\alpha(k)$. If $\sigma = -1$, we use $\omega'_\alpha(k) > 0$ from Lemma 1 to conclude $\Omega'_{-1}(k) <  0$. If $\sigma = 1$ and $|k|>1$, we use $\omega'_\alpha (k) < \alpha |c_0|$ from Lemma 1 to conclude $\Omega'_{1}(k) = -\alpha c_0 + \omega'_\alpha(k) < 0$, since $c_0 > 0$. An analogous proof holds when $c_0<0$. \end{proof}
~\indent In what follows, we consider $c_0>0$, which corresponds to right-traveling Stokes waves. Similar statements hold when $c_0<0$ if one rewrites the collision condition \eqref{B8} as $\Omega_{-1}(k) = \Omega_{1}(k+p) \neq 0$, where $k$ and $p$ are redefined appropriately.
\begin{lemma}{3}
For each $p \in \mathbb{R}$ and $\alpha > 0$, there exists a unique $k(p;\alpha) \in \mathbb{R}$ such that $\Omega_1(k(p;\alpha)) = \Omega_{-1}(k(p;\alpha)+p)$. If $p \in \mathbb{Z}$ and $c_0 > 0$, we have $\cdots < k(1;\alpha) < k(0;\alpha) < k(-1;\alpha) < \cdots$. Moreover, $|k(p;\alpha)| > |p|$ for $p \in \mathbb{Z} \setminus \{0,\pm 1 \}$ and $c_0>0$.
\end{lemma}
\begin{proof} Fix $p \in \mathbb{R}$ and $\alpha > 0$. Define $F(k,p) = \Omega_1(k)-\Omega_{-1}(k+p)$. Then,

\begin{align}
F(k,p) \sim 2k \sqrt{\frac{\alpha}{|k|}} + \mathcal{O}\left(\frac{1}{\sqrt{|k|}}\right) \quad \textrm{as} \quad |k| \rightarrow \infty.
\end{align}

\noindent Since $F$ has opposite signs as $k \rightarrow \pm \infty$, there exists at least one root, denoted $k(p;\alpha)$. Since $\partial_k F(k,p) = \omega'(k)+\omega'(k+p) > 0$ by Lemma 1, $k(p;\alpha)$ is the only root of $F$ in $\mathbb{R}$, proving the first claim of the theorem. \\
\indent To prove the second claim, differentiate $F(k(p;\alpha),p)$ with respect to $p$. Using the definition of $\Omega_{\sigma}$,

\begin{align}
k'(p) &= \frac{\Omega_{-1}'(k(p;\alpha)+p)}{\omega'(k(p;\alpha)) + \omega'(k(p;\alpha)+p)},
\end{align}

\noindent which is well-defined since $\partial_k F(k,p) > 0$. If $c_0>0$, then Lemma 2 implies $k'(p) < 0$. If $p$ is restricted to $\mathbb{Z}$, we have $\cdots < k(1;\alpha) < k(0;\alpha) < k(-1;\alpha) < \cdots$, as desired. \\
\indent To prove the third claim, first consider $p>1$. Suppose $k(p;\alpha) \geq -p$. Since $\omega_\alpha(k)$ is odd and strictly increasing by Lemma 1,

\begin{subequations} \begin{align}\omega_\alpha(k(p;\alpha)) \geq& -\omega_\alpha(p),  \label{B11bb} \\ \omega_\alpha(k(p;\alpha)+p) \geq& \omega_\alpha(0) = 0.  \label{B11c} \end{align}   \end{subequations}

\noindent Using the definition of $\Omega_{\sigma}$, $F(k(p;\alpha),p) = 0$ can be rewritten as

\begin{align}
\omega_\alpha(k(p;\alpha)) + \omega_\alpha(k(p;\alpha)+p) = -\alpha c_0 p. \label{B11d}
\end{align}

\noindent Together with \eqref{B11d}, inequalities \eqref{B11bb} and \eqref{B11c} imply

\begin{align}
-\omega_\alpha(p) \leq -\alpha c_0 p \quad \Rightarrow \quad \frac{\omega_\alpha(p)}{p} \geq \alpha c_0  = \frac{\omega_\alpha(1)}{1},
\end{align}

\noindent a contradiction since $\omega_\alpha(z)/z$ is strictly decreasing for $z>0$. Therefore, $k(p;\alpha) < -p$ for $p>1$. Since $\Omega_{\sigma}(k)$ is odd, $k(p;\alpha) = -k(-p;\alpha)$. Therefore, when $p < -1$, $k(p;\alpha) > -p$. Combining the two cases yields $|k(p;\alpha)| > |p|$ whenever $p \in \mathbb{Z} \setminus \{0,\pm1\}$ and $c_0>0$, as desired. \end{proof}
Lemma 3 has several consequences:
\begin{enumerate}
\item[1.] When $c_0>0$, $k(p;\alpha) < 0$ for $p>0$, and $k(p;\alpha) > 0$ for $p<0$.
\item[2.] When $c_0>0$, $k(p;\alpha) \rightarrow \pm \infty$ as $p \rightarrow \mp \infty$. In fact, the sequence $\{k(p;\alpha)\}$ must grow at least linearly as $|p| \rightarrow \infty$. Formal arguments suggest quadratic growth in this limit.
\item[3.] The products $k(p;\alpha)(k(p;\alpha)+p) >0$ and $\omega_{\alpha}(k(p;\alpha))\omega_{\alpha}(k(p;\alpha)+p)>0$ when $c_0>0$. The latter of these products is related to the Krein signature condition proposed in \cite{mackay86}. In effect, Lemma 3 provides a different proof that collided eigenvalues \eqref{11} have opposite Krein signatures, consistent with \cite{deconinck17}.
\end{enumerate}
The above results lead to the following theorem.

\begin{theorem}{4} Let $c_0 > 0$. If $p \in \{0, \pm 1 \}$, then the collision condition \eqref{B8} is not satisfied. If $p \in \mathbb{Z} \setminus \{0,\pm 1 \}$, then $k(p;\alpha)$ solves the collision condition. Moreover, $\cdots < \lambda_{i,3} < \lambda_{i,2} < 0 < \lambda_{i,-2} < \lambda_{i,-3} < \cdots$, where $\lambda_{i,p}$ is the imaginary part of the collision point corresponding to $k(p;\alpha)$.
\end{theorem}
\begin{proof} When $p = 0$ or $\pm 1$, we have $k(p;\alpha) = 0$ or $\mp 1$, respectively, by inspection. It follows that $\Omega_1(k(p;\alpha)) = 0$ in all three cases, and so \eqref{B8} is not satisfied. This proves the first claim. \\ \indent
To prove the second claim, consider the sequence $\{\Omega_1(k(p;\alpha))\}$, $p \in \mathbb{Z} \setminus \{0,\pm 1 \}$. From Lemma 3, $\{k(p;\alpha)\}$ is a strictly decreasing sequence, and each element of this sequence satisfies $|k(p;\alpha)| > |p| > 1$. Thus, Lemma 2 holds, and the sequence $\{\Omega_1(k(p;\alpha))\}$ is strictly increasing. Since $\Omega_1(\pm 1 ) = 0$, we have $\Omega_1(k(p;\alpha)) \neq 0$. This proves that $k(p;\alpha)$ satisfies the collision condition \eqref{B8} for the relevant values of $p$. \\ \indent
The proof of the third claim is immediate since $\{\Omega_1(k(p;\alpha))\}$ is strictly increasing. \end{proof}
Let $\mu_0 = k(p;\alpha) - [k(p;\alpha)]$ for $p \in \mathbb{Z} \setminus \{0,\pm 1\}$, where $[\cdot]$ denotes the nearest integer function. Then, $\mu_0$ is the unique Floquet exponent in $[-1/2,1/2]$ for which $\lambda^{(1)}_{0,\mu_0,n}$ and  $\lambda^{(-1)}_{0,\mu_0,m}$ satisfy \eqref{11} with $n = [k(p;\alpha)]$ and $m=n+p$.





\end{document}